 \documentclass[11pt]{article}
 \UseRawInputEncoding
 \usepackage{bbm}
\usepackage{amssymb, amsthm, amsmath, amscd}
\usepackage{hyperref}

\newtheorem{thm}{Theorem}[section]
\newtheorem{lem}[thm]{Lemma}
\newtheorem{prop}[thm]{Proposition}
\newtheorem{cor}[thm]{Corollary}
\newtheorem{NN}[thm]{}

\theoremstyle{definition}\newtheorem{df}[thm]{Definition}
\theoremstyle{definition}\newtheorem{rem}[thm]{Remark}
\theoremstyle{definition}\newtheorem{exm}[thm]{Example}

\renewcommand{\phi}{\varphi}

\newcommand{\N}{\mathbb{N}}
\newcommand{\Z}{\mathbb{Z}}

\newcommand{\R}{\mathbb{R}}
\newcommand{\C}{\mathbb{C}}
\newcommand{\T}{\mathbb{T}}
\newcommand{\I}{{\mathbb{I}}}

\newcommand{\hm}{homomorphism}
\newcommand{\dt}{\delta}
\newcommand{\ep}{\epsilon}
\newcommand{\la}{\langle}
\newcommand{\ra}{\rangle}
\newcommand{\andeqn}{\,\,\,{\rm and}\,\,\,}
\newcommand{\rforal}{\,\,\,{\rm for\,\,\,all}\,\,\,}
\newcommand{\CA}{$C^*$-algebra}
\newcommand{\SCA}{$C^*$-subalgebra}

\newcommand{\af}{{\alpha}}
\newcommand{\bt}{{\beta}}
\newcommand{\dist}{{\rm dist}}

\newcommand{\D}{\mathbb D}
\newcommand{\beq}{\begin{eqnarray}}
\newcommand{\eneq}{\end{eqnarray}}
\newcommand{\tforal}{\,\,\,\text{for\,\,\,all}\,\,\,}
\newcommand{\tand}{\,\,\,\text{and}\,\,\,}

\newcommand{\Wlog}{Without loss of generality}
\newcommand{\diag}{{\rm diag}}

\newcommand{\td}{\tilde}

\title{Almost commuting self-adjoint operators   and measurements
}
\author{Huaxin Lin}

\date{}

\begin{document}

\maketitle

\begin{abstract}
We study the problem when an 
$n$-tuple  of self-adjoint operators 
in an infinite dimensional separable Hilbert space $H$ with small commutators 
 is close to an $n$-tuple  of commuting 
self-adjoint operators on $H.$ 
We give an affirmative answer  to the problem when the synthetic-spectrum and 
the essential synthetic-spectrum are close.
Examples are also exhibited that, in general, the answer to the problem  when $n\ge 3$ is negative
even the associated Fredholm index vanishes.
This is an attempt to solve a   problem proposed by David Mumford related to quantum theory and measurements. 
%
\end{abstract}

\section{Introduction}
At the First International Congress of Basic Science held in 2023,  David Mumford delivered 
the opening  plenary lecture on {\it Consciousness, Robots and DNA} (\cite{Mf}).
He began by  cautioning  the audience  that there would be no mathematics in his talk.
Nevertheless, by the very end of this magnificent lecture---which culminated with  the idea of 
DNA as  a measuring instrument opening a Pandora's box--- he proposed an unexpected  problem in
\CA s: determining  when almost-commuting self-adjoint operators can be  approximated by commuting ones.

A version of this  problem proposed  by Mumford may be reformulated as follows:
Let $\ep>0$ and $n\in \N$ be a positive integer.
 When is there a constant $\dt>0$ such that  the following statement holds?
 For any  separable  Hilbert space  $H$ and  self-adjoint operators 
 $T_1, T_2,...,T_n$ (with $\|T_j\|\le 1$) satisfying
 \beq\label{mfp-1}
 \|T_iT_j-T_jT_i\|<\dt,\,\,\, 1\le i,j\le n,
 \eneq
there exist self-adjoint operators $S_1,S_2,...,S_n$ on $H$
such that
\beq\label{mfp-2}
S_iS_j=S_jS_i\andeqn \|S_j-T_j\|<\ep,\,\, \, 1\le i,j\le n.
\eneq

In the case that $n=2$ and $H$  is any finite dimensional Hilbert space  (no bound on the dimension)
 the  exactly the same problem is known 
as von-Neumann-Kadison-Halmos problem for almost commuting self-adjoint matrices
(see \cite{Halmos1} and \cite{Halmos2}).  In this setting, the problem has an affirmative solution 
\cite{Linmatrices} (see also \cite{FR}  and  \cite{Has}, and \cite{LT} for  fixed dimensions).
However, in an infinite dimensional Hilbert space $H,$  as Mumford noted in his recent book (p.182) --
and as was previously known -- the answer is  negative  in general (see also Example 4.6 of \cite{FR}). 

 The affirmative solution for almost commuting self-adjoint matrices has spurred further research in 
 the study 
 of weak semi-projectivity in \CA s. It also has  a deep impact on the Elliott 
program of classification 
of amenable \CA s (see \cite{EGLP}, \cite{LinTAMS99}, \cite{Linann03} and \cite{Linduke}). 
Moreover,  Mumford's problem opens  new avenues for applying this branch of \CA\, theory 
to previously unexpected fields, 
such as DNA research and artificial intelligence (\cite{Mf2}).

The first main result  of this paper is  as follows:
 
 \begin{thm}\label{TTTmul}
Let $n\in \N$ and  $\ep>0.$  There exists $\dt(n, \ep)>0$ satisfying the following:
Suppose that $H$ is an infinite dimensional separable Hilbert space
and 
$T_1, T_2,...,T_n$ are  self-adjoint linear operators on $H$ with $\|T_i\|\le 1$ 
($1\le i\le n$) such that
\beq\label{TTmul00-1}
\|T_iT_j-T_jT_i\|<\dt,\,\,\, 1\le i,j\le n\tand
d_H(X, Y)<\ep/8,
\eneq
where $X=s{\rm Sp}^{\ep/8}(T_1, T_2,...,T_n)$ and $Y=s{\rm Sp}^{\ep/8}_{ess}((T_1, T_2,...,T_n)).$
Then there are bounded self-adjoint  linear operators $S_1, S_2,...,S_n$ on $H$
such that
\beq
S_iS_j=S_jS_i,\,\,\, 1\le i,j\le n\tand \|T_j-S_j\|<\ep,\,\,\, 1\le j\le n.
\eneq
\end{thm}

\noindent
(See Remark \ref{Rcomp} for the case that each $T_j$ is compact.)

Here   $s{\rm Sp}^\eta((T_1, T_2, ...,T_n))$ and 
$s{\rm Sp}_{ess}^\eta((T_1, T_2,...,T_n))$ are   the  sets of $\eta$-synthetic-spectrum, 
and essential $\eta$-synthetic-spectrum, respectively
(see Definition \ref{Dpseudosp}).
In \eqref{TTmul00-1}, $d_H(X,Y)$ is the Hausdorff distance between the two compact subsets 
of $\I^n,$ the $n$-dimensional unit cube.  In the special case 
that $s{\rm Sp}^{\ep/8}_{ess}((T_1, T_2, ...,T_n))$ is full, i.e., it contains the whole $\I^n,$
then the second part of \eqref{TTmul00-1} is automatically satisfied.

Remark \ref{Finalrem} provides some justification  of the  second condition 
related to macroscopic observables and  measurements, and a discussion of  
the Mumford problem.

More generally, let $H$ be a Hilbert (right) module over a \CA\, $A,$ $L(H)$ be 
the \CA\, of all bounded module maps on $H$  with adjoints and let $F(H)$ be the linear span 
of rank one module maps $T$ of the form $T(x)=z\la y,x\ra $ (for all $x\in H$), where 
$y, z\in H.$  Denote by $K(H)$ the closure of $F(H)$ which is a 
\CA.   We also have the following:

\begin{thm}\label{TTTmodule}
Let $n\in \N$ and  $\ep>0.$  There exists $\dt(n, \ep)>0$ satisfying the following:
Suppose that  $H$ is a countably generated Hilbert-module over a $\sigma$-unital  purely infinite simple \CA\,
$A$ and  $T_1, T_2,...,T_n\in L(H)$ are self-adjoint bounded module maps with $\|T_i\|\le 1$ 
($1\le i\le n$) such that
\beq\label{TTmul000-1}
\|T_iT_j-T_jT_i\|<\dt,\,\,\, 1\le i,j\le n.
\eneq
Suppose also

\noindent 
 (i)   $K(H)=L(H),$
 or 
 
 \noindent
(ii)   $K(H)\not=L(H),$ and 
$
d_H(X, Y)<\ep/8,$
where $X=s{\rm Sp}^{\ep/8}(T_1, T_2,...,T_n)$ and\\ $Y=s{\rm Sp}^{\ep/8}((\pi_c(T_1), \pi_c(T_2),...,\pi_c(T_n))),$
and $\pi_c: L(H)\to L(H)/K(H)$ is the quotient map.

Then there are  $S_1, S_2,...,S_n\in L(H)_{s.a.}$
such that
\beq
S_iS_j=S_jS_i,\,\,\, 1\le i,j\le n\tand \|T_j-S_j\|<\ep,\,\,\, 1\le j\le n.
\eneq
\end{thm}

While we are mostly interested in the case that $n>2,$ we also looks the special case when 
$n=2.$ 
For a pair of self-adjoint operators, 
we have a more definite answer
for some Hilbert modules.

\begin{thm}\label{MT-pair}
Let $\ep>0.$ There exists $\dt(\ep)>0$ satisfying the following:
Suppose that\\
(i) $H$ is an infinite dimensional separable  Hilbert space, or \\
(ii) $H$ is a countably generated Hilbert module over a $\sigma$-unital  purely infinite simple \CA\, $A$
such that  $L(H)$ has real rank zero, or \\
  (iii) $H$ is  a countably generated Hilbert module over a $\sigma$-unital simple
\CA\, $B$ of stable rank one such that $L(H)$ has real rank zero and $L(H)/K(H)$ is simple.

Suppose  also
that $T_1, T_2\in L(H)$ 
are self-adjoint bounded module maps on $H$  with $\|T_i\|\le 1$ ($i=1,2$) such that
\beq\label{MT-pair-00}
\|T_1T_2-T_2T_1\|<\dt\tand \kappa_1(\lambda-\pi(T_1+iT_2))=0
\eneq
for all $\lambda\not\in s{\rm Sp}^\dt((\pi(T_1), \pi(T_2))),$
where $\pi: L(H)\to L(H)/K(H)$ is the quotient map.
Then there are self-adjoint bounded module maps $S_1, S_2\in L(H)$
such that 
\beq
S_1S_2=S_2S_1\tand \|S_j-T_j\|<\ep,\,\,\, j=1,2.
\eneq
Moreover, the condition that $L(H)$ has real rank zero can be removed if we assume 
that\\ $s{\rm Sp}^\dt((\pi(T_1), \pi(T_2)))$ is connected.
\end{thm}

The second condition in \eqref{MT-pair-00} means that a Fredholm index vanishes (see Definition \ref{Dind}).
Note that both cases (ii) and  (iii)  generalize case (i). 
In case (ii), $K(H)$ is stable (whenever $H$ is not a finitely generated projective Hilbert module),
and $L(H)/K(H)$ is simple (analogous to the Calkin algebra).  
Moreover, in case (ii),  
$L(H)/K(H)$ retains its simiplicity while 
$L(H)$ has real rank zero.  

Case (i) in Theorem \ref{MT-pair} is a  more definite result  than that of  Theorem \ref{TTTmul}
for $n=2.$  In fact, 
we prove a slightly more general result (see Theorem \ref{MT-pair+} below).
One should also compare (i) in Theorem \ref{MT-pair} to Theorem 1.1 
of \cite{KS} (see Remark \ref{RR}).   Theorem 1.1 of \cite{KS}
provides a more  definitive estimate of ${\rm dist}(a, N_f(A)),$ 
the distance from an element 
to the set of normal elements with finite spectrum in \CA s with real rank zero. 
However, 
when  $K_1(K(H))\not=0,$  normal elements $N=S_1+iS_2$ may not be 
approximated by normal elements  in $K(H)$ with finite spectrum. 
Thus,  cases (ii) and (iii) in Theorem \ref{MT-pair} cover scenarios where 
almost normal elements are close to normal elements that cannot be approximated 
normal elements with finite spectra --- complementing the results in \cite{KS}. 

We further demonstrate   that 
the vanishing  index  condition for the 
operator pair  is necessary. 
Given 
well-known instablity of the 
spectrum (and essential spectrum) of bounded operators under perturbations, it becomes essential to 
consider some version 
of $\dt$-spectrum (typically  larger than the spectrum).  For $n\ge 2,$ we introduce the concept of 
$\dt$-synthetic-spectrum for $n$-tuples of self-adjoint elements in a unital \CA\, 
(see 
Definition \ref{Dpseudosp}). 
%
%
We will also show that,  unfortunately, the  version  of Mumford's problem  formulated  above has a negative solution in general--- 
even the associated Fredholm index vanishes when $n>2.$

In quantum mechanics, one often encouters  the statement,  ``when the commutators go to zero, we recover 
the classical system".   If ``recover" means the observables are approximated by commuting 
observables, then
example in (iii) of Chapter 14 of \cite{Mf2} (Example 4.6 of \cite{FR}),  Proposition \ref{Lnn-0}  and Proposition \ref{PLast} all reveal significnat topological obstructions to such claims.
Nevertheless Theorem \ref{TTTmul}, interpreted as  
in Remark \ref{Finalrem}, may  offer 
an alternative perpective.

The paper is organized as follows.
Section 2 contains some notations used in the paper and 
gives the definition of $\dt$-synthetic-spectrum and $\dt$-near-spectrum for $n$-tuples 
of self-adjoint operators, and discusses the relationship between them (see also Proposition \ref{Pesspsp}). 
Section 3 is a study of $n$-tuples of almost commuting self-adjoint elements 
in a unital purely infinite simple \CA. 
Section 4 presents the proof of  Theorem \ref{TTTmul} and Theorem \ref{TTTmodule}.
In Section 5, we will discuss  the case of a pair of almost commuting self-adjoint  bounded module 
maps and provide a proof of 
 Theorem \ref{MT-pair}.
 %
%
In section 6,  the last section, 
we show why the Fredholm index does not appear in the statement of Theorem \ref{TTTmul}.
We then present an example that an $n$-tuple version of Theorem \ref{MT-pair} 
does not hold in general. A final remark was added which  contains the justification of 
the second condition in \eqref{TTmul00-1} and its connection to the Mumford original question.

\section{Preliminaries}

\begin{df}\label{D1}
Let $A$ be a \CA. Denote by $A_{s.a.}$ the set of all self-adjoint elements in $A.$
Denote by $M(A)$ the multiplier algebra (the idealizer of $A$ in $A^{**}$). 
If $x, y\in A$ and $\ep>0,$ we write 
\beq
x\approx_\ep y,\,\,\, {\rm if}\,\,\, \|x-y\|<\ep.
\eneq

For unital \CA\, $A,$ dnote by $GL(A)$ and $GL_0(A),$ the group of invertible elements and 
the path connected component of $GL(A)$ containing the identity of $A,$ respectively.
\end{df}

\begin{df}\label{Dpic}
Let $H$ be an infinite dimensional separable Hilbert space.
Denote by $B(H)$ the \CA\, of all bounded linear operators on $H$ 
and by ${\cal K}$ the \CA\, of all compact operators on $H.$
Denote by $\pi_c: B(H)\to B(H)/{\cal K}$ the quotient map. 

Let $T\in B(H).$ Recall that the essential spectrum of $T,$ ${\rm sp}_{ess}(T)={\rm sp}(\pi_c(T)),$
is the spectrum of $\pi_c(T).$
\end{df}

\begin{df}\label{Dmodule}
Let $A$ be a \CA\, and $H$ a Hilbert (right)  module  over $A$ (or Hilbert $A$-module; see \cite{Kas}).
Denote by $L(H)$ the \CA\, of all bounded module 
maps on $H$ with adjoints.  A rank-one module map 
$T\in L(H)$ is a bounded module map of the form 
$T(x)=z\la y,x\ra $ for all $x\in H$ (and fixed $y, z\in H$), where 
$\la \cdot, \cdot\ra$ is the $A$-valued inner product on $H.$ 
Denote by $F(H)$ the linear span of rank-one module maps and $K(H)$
the (norm) closure of $F(H).$  Following Theorem 1 of \cite{Kas}, 
we identify $L(H)$ with $M(K(H)).$
\end{df}

\begin{df}\label{Dind}
Let $A$ be a unital \CA\, and $x\in A$ be an invertible element.
Define $u_x=x(x^*x)^{-1/2}.$ Note that $u_x$ is a unitary.
Define 
\beq
\kappa_1(x)=[u_x]\in K_1(A).
\eneq
Then $\kappa_1(x)=0$ if and only if $x\in GL_0(A),$ when $K_1(A)=U(A)/U_0(A).$
Let $T\in B(H)$ be such that $\pi_c(T)$ is invertible in $B(H)/{\cal K}.$
Then 
\beq
{\rm Ind}(T)=\kappa_1(\pi_c(T))=[u_{\pi_c(T)}]\,\,\, {\rm in}\,\,\, K_1(B(H)/{\cal K})\cong\Z.
\eneq
\end{df}

\begin{df}
In $\R^n,$ by ${\rm dist}(x,y),$ we mean the euclidian distance between $x$ and $y,$
i.e., $\|x-y\|_2.$
Denote by $\I^n=\{(r_1, r_2, ...,r_n)\in \R^n: |r_i|\le 1\}.$ If $\eta>0$ and $x\in \I^n,$ define\\
$B(x, \eta)=\{y\in \I^n: {\rm dist}(x, y)<\eta\}.$
Let $e_0(\xi)=1$ for all $\xi\in \I^n$ be the constant function,
$e_i\in C(\I^n)$ be defined by
$e_i((r_1, r_2,...,r_n))=r_i$ for $(r_1, r_2,...,r_n)\in \I^n,$ $i=1,2,...,n.$
{\bf This notation will be used throughout this paper.}
Note that $C(\I^n)$ is generated by $\{e_i: 0\le i\le 1\}.$


\end{df}

\begin{df}\label{n-net}
Fix $M\ge 1.$
Fix an integer $k\in \N.$   Define
%
$$P_k=\{\xi=(x_1,x_2,...,x_n)\in \R^n: x_j=m_j/k,  |m_j|\le M k, m_j\in \Z,\,\,1\le j\le n\}.$$
$P_k$ has only finitely many points.  Most of the time, we choose $M=1.$
\end{df}

\begin{df}\label{Dpseudosp}
Let $n\in \N$ and $M>0.$ 
In what follows,  for each $0<\eta<1,$ we  choose a fixed integer $k=k(\eta)\in \N$
such that $k=\inf\{l\in \N: (M+1)/l<{\eta\over{1+2\sqrt{n}}}\}.$ Denote $D^\eta=P_k.$
We write $D^\eta=\{x_1,x_2,...,x_m\}.$ Then $D^\eta$ is $\eta/2$-dense  in $M^n=\{(t_1,t_2,...,t_n)\in \R^n: 
|t_i|\le M\}.$
Moreover, $D^\eta\subset D^\dt$ if $0<\dt<\eta.$

Let $A$ be a unital \CA\, and $(a_1,a_2,...,a_n)$ be an $n$-tuple  of self-adjoint elements in 
$A$ with $\|a_i\|\le M$ ($1\le i\le n$). 
Fix $\xi=(\lambda_1, \lambda_2,...,\lambda_n)\in \I^n.$
Let 
$\theta_{\lambda_i,\eta}\in C([-M, M])$ be such that 
$0\le \theta_{\lambda_i, \eta}\le 1,$ $\theta_{\lambda_i, \eta}(t)=1,$ if $|t-\lambda_i|\le 3\eta/4,$ 
$\theta_{\lambda_i,\eta}(t)=0$
if $|t-\lambda_i|\ge  \eta,$  and $\theta_{\lambda_i \eta}$ is linear in $(\lambda_i-\eta, \lambda_i-3\eta/4)$
and in $(\lambda_i+3\eta/4, \lambda_i+\eta),$
 $i=1,2,...,n.$ 
 Define, for $(t_1, t_2,...,t_n)\in M^n,$
\beq
\Theta_{\xi, \eta}(t_1,t_2,...,t_n)&=&\prod_{i=1}^n\theta_{\lambda_i, \eta}(t_i)\andeqn\\\label{Dpsp-5}
\Theta_{\xi, \eta}(a_1,a_2,...,a_n)&=&\theta_{\lambda_1,\eta}(a_1)\theta_{\lambda_2,\eta}(a_2)\cdots 
\theta_{\lambda_n,\eta}(a_n).
\eneq
Note that  we do not assume that $a_1, a_2,...,a_n$ mutually commute and  the product in \eqref{Dpsp-5} has 
a fixed order.

For $x_j\in  D^\eta,$  write $x_j=(x_{j,1}, x_{j,2},...,x_{j,n})\in D^\eta.$
We may  sometime write $\theta_{j,i, \eta}:=\theta_{x_{j,i},\eta}$ and 
$\Theta_{j, \eta}:=\Theta_{x_j, \eta},$ $1\le i\le n,$ $j=1,2,...,m.$
%
%
Set 
\beq
s{\rm Sp}^\eta((a_1,a_2,..., a_n)=
\bigcup{}_{_{\|\Theta_{x_j, \eta}(a_1,a_2,...,a_n)\|\ge 1-\eta}} \overline{B(x_j, \eta)}. 
\eneq
This  union of finitely many closed balls
$s{\rm Sp}^\eta((a_1,a_2,...,a_n))$ is called $\eta$-synthetic-spectrum of the $n$-tuple 
$(a_1,a_2,...,a_n)$ which, by the definition, is compact.  

Let $0<d<\eta.$ Suppose $x_j\in D^\eta.$ Then $x_j\in D^d.$
Note that $\theta_{j,i, d}\le \theta_{j, i,\eta}$ ($1\le i\le n$). 
Hence $\Theta_{j, d}\le \Theta_{j,\eta}.$ 
It follow that, if $\|\Theta_{j,d}(a_1,a_2,...,a_n)\|\ge 1-d,$ then 
$\|\Theta_{j,\eta}(a_1,a_2,...,a_n)\|\ge \|\Theta_{j,d}(a_1,a_2,...,a_n)\|\ge 1-d>1-\eta.$
Then, if $0<d<\eta,$
\beq\label{pSp<}
s{\rm Sp}^{d}((a_1,a_2,...,a_n))=\bigcup{}_{_{\|\Theta_{j, \eta}(a_1,a_2,...,a_n)\|\ge 1-\eta}} \overline{B(x_j, \eta)}
\subset s{\rm Sp}^\eta((a_1, a_2,...,a_n)).
\eneq
Note also that, if $d\ge 2\dt,$  then
\beq\label{pSpdteta}
\{x\in M^n: {\rm dist}(x,s{\rm Sp}^{\dt/2}((a_1, a_2,...,a_n)))<\dt/2\}\subset s{\rm Sp}^{d}((a_1,a_2,...,a_n)).
\eneq
%
To see this, 
let ${\rm dist}(x, s{\rm Sp}^{\dt/2}((a_1, a_2,...,a_n)))<\dt/2.$ 
There is $y_i\in D^{\dt/2}$ with $\|\Theta_{y_i,\dt/2}(a_1,a_2,...,a_n)\|\ge 1-\dt/2$
such that $x\in \overline{B(y_i, \dt/2)}.$
There is $x_j\in D^{d}$ such that ${\rm dist}(x_j, y_i)<d/2.$
Write $y_i=(\lambda_1,\lambda_2,...,\lambda_n)$ and 
$x_j=(\lambda_1', \lambda_2',...,\lambda_n').$
If $|t_l-\lambda_l|<\dt/2,$ then 
$$|t_l-\lambda_l'|\le \dt/2+|\lambda_l-\lambda_l'|<\dt/2+d/2<3d/4,$$ $1\le l\le n.$ Therefore
$\Theta_{y_i, \dt/2}\le \Theta_{x_j, d}.$ 
Hence $\|\Theta_{j, 2\dt}(a_1,a_2,...,a_n)\|\ge 1-\dt/2\ge 1-d.$ 
Therefore $x\in B(x_j, \dt+\dt/2)\subset s{\rm Sp}^{d}((a_1,a_2,...,a_n)).$

In what follows, we are particularly interested in the case that $M=1.$

If $A=B(H)$ for some infinite dimensional separable Hilbert space, 
define 
\beq
s{\rm Sp}^\eta_{ess}((a_1,a_2,...,a_n))=s{\rm Sp}^\eta((\pi_c(a_1), \pi_c(a_2),...,\pi_c(a_n)))
\eneq
which will be called essential $\eta$-synthetic-spectrum.
It is clear that 
\beq
s{\rm Sp}^\eta_{ess}((a_1,a_2,...,a_n))\subset s{\rm Sp}^\eta((a_1,a_2,...,a_n)).
\eneq
\end{df}

\begin{df}\label{Dcpc} 
Let $A$ and $B$ be \CA s.  A linear map $L: A\to B$
is a c.p.c. map if it is completely positive and contractive.

\end{df}

\begin{df}\label{Dappsp}
Let $A$ be a  unital  \CA\, and $a_1,a_2,...,a_n\in A_{s.a.}$ for some $n\in \N.$
Let us assume that $\|a_i\|\le 1,$ $i=1,2,...,n.$ 
Fix $0<\eta<1/2.$ Suppose  that $X\subset \I^n$ is a compact subset 
and $L: C(X)\to A$ is a unital c.p.c. map
such that 
\beq
%
\hspace{-2.1in}{\rm (i)} &&\|L(e_j|_X)-a_j\|<\eta,\,\,\, 1\le j\le n, \\
\hspace{-2.1in}{\rm (ii)}&&\|L((e_ie_j)|_X)-L(e_i|_X)L(e_j|_X)\|<\eta, \,\,\, 1\le i,j\le n,\andeqn\\
\hspace{-2.1in}{\rm (iii)} &&  \|L(f)\|\ge 1-\eta
\eneq
for any $f\in C(X)_+$  which has 
value $1$ on an open ball with the center $x\in X$ (for some $x$) and the radius $\eta.$
Then we say that $X$ is an $\eta$-near-spectrum of  the $n$-tuple 
$(a_1, a_2,...,a_n).$ 
We may write 
$$
X=nSp^\eta((a_1,a_2,...,a_n))
$$
for convenince. The reader is recommended to read Remark \ref{Rsp} for clarification.

If $L$ is, additionally,  a unital {\it \hm,}  then  we say $X$ is 
an $\eta$- spectrum of the $n$-tuple $(a_1,a_2,...,a_n).$

Let $H$ be an infinite dimensional separable Hilbert space and let 
$T_1, T_2,...,T_n\in B(H)_{s.a.}.$ 
If $X$ is an $\eta$-spectrum of $(\pi_c(T_1), \pi_c(T_2),...,\pi_c(T_n)),$
we say that $X$ is an  essential $\eta$-spectrum for $(T_1, T_2,...,T_n).$
In this case we write 
$$
Sp_{ess}^\eta((a_1,a_2,...,a_n)):=X.
$$

\end{df}

\begin{rem}\label{Rsp}   Here are some clarifications.
%

1. 
It should be noted that $nSp^\dt((a_1, a_2,...,a_n))$ is not uniquely determined and 
depends on the choice of the c.p.c.~map.  
Writing $X=nSp^\dt((a_1,a_2,...,a_n))\not=\emptyset$ {\it simply means that there exists a   non-empty 
compact subset $X$ and a unital c.p.c. map $L: C(X)\to A$ satisfying  conditions (i), (ii) and (iii) of Definition \ref{Dappsp}.}
This convention is adopted for notational convenience.
Moreover, if we write $n{\rm Sp}^\eta((a_1, a_2,...,a_n))=\emptyset,$
we mean that no such $X$ and $L$ exist which satisfy (i), (ii) and (iii).

%
%

2. 
However, we  will later show that $\dt$-near spectra are {\it unique} up to an $\eta$- neighborhood when the
$n$-tuple selfadjoint operators are almost commuting within $\dt$ as elaborated in Proposition \ref{Pspuniq}.
%

3.
On the other hand, 
the $\dt$-synthetic-spectrum is always  uniquely defined and---importantly--
computable in the context of computations.
Their relationship is described in Proposition \ref{Pspuniq} below (see also Proposition \ref{Pesspsp}).

4. Let $X$ be an $\eta$-near-spectrum  and 
$Y$ be  a $\dt$-near -spectrum  for $(a_1,a_2,...,a_n),$  respectively.
Suppose that $\dt<\eta,$ then, by the definition, 
$Y$ is also an $\eta$-near-spectrum. 
 In particular, if $nSp^\dt((a_1,a_2,...,a_n))\not=\emptyset,$ then 
$nSp^\eta((a_1,a_2,...,a_n))\not=\emptyset.$

5.
In practice, computing the spectrum of an operator 
$T_1+i T_2$ (where $T_1, T_2\in B(H)_{s.a.}$) may be more challenging  than computing  the norm 
of $\Theta_{j, \eta}((T_1, T_2)).$ 
%
\end{rem}

\begin{prop}\label{Pappsp}
Fix $n\in \N.$ For any $\eta>0,$ there exists $\dt(n,\eta)>0$ satisfying the following:
Suppose that $A$ is a  unital \CA\, and $a_i\in A_{s.a.}$ with $\|a_i\|\le 1,$ $1\le i\le n,$
such that
\beq
\|a_ia_j-a_ja_i\|<\dt,\,\,\, 1\le i,j\le n.
\eneq
Then 

(1) 
$X:=s{\rm Sp}^\eta((a_1,a_2,...,a_n))\not=\emptyset$ and,

(2) $Y:=nSp^\eta((a_1,a_2,...,a_n))\not=\emptyset,$  i.e., there exists 
a non-empty compact subset $Y\subset \I^n$ and a c.p.c. map
$L: C(Y)\to A$ satisfies condition (i), (ii) and (iii) in Definition \ref{Dappsp}.

%
%
\end{prop}

\begin{proof}
Suppose that the lemma is false. 
Then we obtain an $\eta>0$ and a  sequence of unital \CA s $\{A_k\}$ and 
a sequence of $n$-tuples $\{a_{j,k}\}\subset \{(A_k)_{s.a.}\}$ with $\|a_{j,k}\|\le 1$
($1\le j\le n$), $k\in \N$ such that
\beq
\lim_{k\to\infty}\|a_{i,k}a_{j,k}-a_{j,k}a_{i,k}\|=0,\,\,\, 1\le j\le n,
\eneq
and, 
for case (1),
$s{\rm Sp}^\eta((a_{1, k}, a_{2,k},...,a_{n,k}))=\emptyset,$ $k\in \N,$
(or, for case (2),
$nSp^\eta((a_{1,k},a_{2,k},...,a_{n,k}))=\emptyset,$ $k\in \N.$)

Consider $B=\prod_kA_k$ and quotient map $\Pi: B\to B/\bigoplus_k\{A_k\}).$ 
Let $s_j=\Pi(\{a_{j,k}\}),$ $j=1,2,...,n.$
Then $s_is_j=s_js_i,$ $1\le i,j\le n.$
Let $C$ be the \SCA\, of $B/\bigoplus_k A_k$ generated by $s_0=1, s_1,s_2,...,s_n.$
Then $C=C(X_0)$ for some non-empty compact subset $X_0$ of $\I^n.$ 
Let $\phi: C(X_0)\to B/\bigoplus_kA_k$ be the monomorphism given by $C$
with
 $\phi(e_j|_X)=s_j,$ $j=0,1,2,...,n.$

To show (1),
consider $D^\eta=\{x_1,x_2,..., x_m\}$ (see Definition \ref{Dpseudosp}).
Fix $x\in X_0.$ Choose $x_j\in D^\eta$ such that   ${\rm dist}(x_j, x)<\eta/2,$ 
then $\|\Theta_{x_j, \ep}|_{X_0}\|=1.$ Hence, since $\phi$ is a monomorphism,  
\beq
\|\phi(\Theta_{x_j,\eta}|_{X_0})\|=1.
\eneq
Recall that $\Pi(\{\Theta_{x_j,\eta}(a_{1,k}, a_{2,k},...,a_{n,k}))=\phi(\Theta_{x_j,\eta}|_X).$ Hence
there   is an infinite subset $S\subset  \N$ such that, for all $k\in S,$
\beq
\|\Theta_{x_j,\eta}(a_{1,k}, a_{2,k},...,a_{n,k})\|\ge 1-\eta.
\eneq
It follows that  $X_0\subset \bigcup_{\|\Theta_{x_j,\ep}(a_{1,k},a_{2,k},...,a_{n,k})\|\ge 1-\eta}B(x_j, \eta)$
for all $k\in S.$
In other words, $(a_{1,k}, a_{2,k},...,a_{n,k})$ has $\eta$-synthetic-spectrum containing $X_0\not=\emptyset$ for all $k\in S.$  This contradicts ``$s{\rm Sp}^\eta((a_{1, k}, a_{2,k},...,a_{n,k}))=\emptyset,$ $k\in \N$". Thus (1) follows.

For (2), 
we obtain, by  the Choi-Effros Theorem  (\cite{CE}),  a c.p.c.  map
$L: C(X_0)\to B$ such that $\Pi\circ L=\phi.$ 
Write $L=\{L_k\},$ where each $L_k: C(X_0)\to A_k$ is a c.p.c map, $k\in \N.$
Then 
\beq\label{Pappsp-10}
\lim_{k\to\infty}\|L_k(e_j|_{X_0})-a_{j,k}\|=0.
\eneq

Let 
$\{x_1, x_2,...,x_l\}\subset X_0$ be an $\eta/4$-dense subset, and 
$f_1, f_2,...,f_l\in C(X_0)_+$ with $0\le f_j\le 1,$ $f_j(x)=1$ if ${\rm dist}(x, x_j)\le \eta/4$
and $f_j(x)=0$ if ${\rm dist}(x, x_j)\ge \eta/2,$ $j=1,2,...,l.$

Since $\Pi\circ L=\phi$  is injective, 
there exists a subsequence $\{k(l)\}\subset \N$ such that
\beq\label{Pappsp-11}
\|L_{k(l)}(f_j)\|\ge (1-\eta/4),\,\,\,j=1,2,...,m.
\eneq
Let  $f\in C(X_0)_+$ and $x\in X_0$  be such that $f(y)\ge 1$ for all $y\in \overline{B(x, \eta)}.$
Since $\{x_1,x_2,...,x_l\}$ is $\eta/4$-dense in $X_0,$ 
there is $x_j$ such that ${\rm dist}(x, x_j)<\eta/4.$   Then $f\ge f_j.$ 
It follows that 
\beq\label{Pappsp-12}
\|L_{k(l)}(f)\|\ge \|L_{k(l)}(f_j)\|\ge (1-\eta/4),\,\,\, j=1,2,...,m.
\eneq
Thus, since $\{e_j|_{X_0}: 0\le j\le n\}$
generates $C(X_0),$ by \eqref{Pappsp-10} and \eqref{Pappsp-12},
 for all large $k(l),$ $X_0$ is an $\eta$-near-spectrum of $(a_{1,k},a_{2,k},...,a_{n,k}).$
A contradiction. 
\end{proof}

\begin{lem}\label{Lpurb}
Let $X$ be a compact metric space and ${\cal G}\subset C(X)$ be 
a finite generating subset.

Then, for any $\ep>0$ and any finite subset ${\cal F}\subset C(X),$ 
there exists $\dt({\cal F}, \ep)>0$ satisfying the following:
Suppose that $A$ is a unital \CA\, and $\phi_1, \phi_2: C(X)\to A$ are 
unital c.p.c. maps such that
\beq
\|\phi_1(g)-\phi_2(g)\|<\dt \andeqn 
\|\phi_i(gh)-\phi_i(g)\phi_i(h)\|<\dt \tforal g, h\in {\cal G},\,\,\, i=1,2.
\eneq
Then, for all $f\in {\cal F},$ 
\beq
\|\phi_1(f)-\phi_2(f)\|<\ep.
\eneq
Moreover, if $X\subset \I^n$ and  ${\cal G}=\{1, e_1|_X,e_2|_X,...,e_n|_X\},$
then, for any $\ep>0$ and 
any finite subset  ${\cal H}\subset C(\I^1),$ 
we may also require
that, for all $h\in {\cal H},$ 
\beq
\|\phi_1(h(e_j)|_X)-h(\phi_1(e_j|_X))\|<\ep,\,\,\, 1\le j\le n.
\eneq
\end{lem}

\begin{proof}
Suppose the lemma is false. Then there exist $\ep_0>0,$ 
a finite subset ${\cal F}\subset C(X)$ and 
a sequence of unital \CA s $\{A_n\},$ and two sequences 
of unital c.p.c. maps 
$\phi_{1,n}, \phi_{2, n}: C(X)\to A_n$ such that
\beq\label{Lpur-5}
&&\lim_{n\to\infty}\|\phi_{1,n}(g)-\phi_{2, n}(g)\|=0\rforal g\in {\cal G},\\\label{Lpur-6}
&&\lim_{n\to\infty}\|\phi_{i,n}(gh)-\phi_{i,n}(g)\phi_{i,n}(h)\|=0\rforal g,h\in {\cal G},\,\,i=1,2,\andeqn\\\label{Lpur-7}
&&\max\{\|\phi_{1,n}(f)-\phi_{2,n}(f)\|: f\in {\cal F}\}\ge \ep_0\rforal n\in \N.
\eneq
Moreover,  in the case that  $X\subset \I^n$ and ${\cal G}=\{e_j|_X: 0\le j\le n\},$ 
there exists $\ep_1$ and a finite subset ${\cal H}_0$ such that  
\beq
\sup\{\|\phi_{1,n}(h(e_j))-h(\phi_{1,n}(e_j))\|: h\in {\cal H}_0,1\le j\le n\}\ge \ep_1
\eneq
for all $n\in \N.$

Put  $B=\prod_nA_n.$ Let $\Pi: B\to B/\bigoplus_nA_n$ the  quotient map. 
Then, since ${\cal G}$ is a generating set of $C(X),$ by \eqref{Lpur-6},
 $\psi_i:=\Pi\circ (\{\phi_{i,n}\}): C(X)\to B/\bigoplus_nA_n$ is a unital \hm,
$i=1,2,$ and by \eqref{Lpur-5}, 
$\psi_1=\psi_2.$ 
Hence 
\beq
\psi_1(f)=\psi_2(f)\rforal f\in C(X).
\eneq
But this implies that, for all sufficiently large $n,$
\beq
\|\phi_{1,n}(f)-\phi_{2,n}(f)\|<\eta_0/2\rforal f\in {\cal F}.
\eneq
This leads a contradiction to \eqref{Lpur-7}.  The first part of the lemma then follows.

For the second part, since $\psi_1$ is a unital \hm,
for any $h\in {\cal H}_0,$
\beq
\psi_1(h(e_j|_X))=h(\psi_1(e_j|_X)),\,\,\, 1\le j\le n.
\eneq
It follows that, for all large $n$ and $h\in {\cal H}_0,$ 
\beq
\|\phi_{1,n}(h(e_j|_X))-h(\phi_{1,n}(e_j|_X)\|<\ep_1,\,\,\,  1\le j\le n.
\eneq
Another contradiction. So the lemma follows.
\end{proof}

\begin{cor}\label{Ctheta}
Fix $n\in \N$ and $0<\eta<1.$ 
There exists $\dt(n, \eta)>0$ satisfying the following:
Suppose that $A$ is a unital \CA,
$a_1, a_2,...,a_n\in A_{s.a.}$ with $\|a_i\|\le 1$ ($1\le i\le n$), $Y\subset \I^n$ is a 
compact subset,  and $L: C(Y)\to A$ is a unital c.p.c. map
such that
\beq\label{Ctheta-1}
\|L(e_ie_j|_Y)-L(e_i|_Y)L(e_j|_Y)\|<\dt\andeqn \|L(e_i|_Y)-a_i\|<\dt, \,\,\, i,j=1,2,...,n.
\eneq
Then, for  $D^\eta=\{x_1, x_2,...,x_m\},$
\beq
\|L(\Theta_{x_j, \eta}|_Y)-\Theta_{x_j, \eta}(a_1,a_2,...,a_n)\|<\eta, \,\,\, 1\le j\le m.
\eneq
\end{cor}

\begin{proof}
Let  $D^\eta=\{x_1,x_2, ...,x_m\}\subset \I^n$ be as defined in \ref{Dpseudosp},
where $x_j=(x_{j,1}, x_{j,2},...,x_{j,n}\}$ ($1\le j\le m$).
There is $\dt_1(\eta)>0$ 
satisfying the following: for any unital \CA\, $A$ and any 
pair of self-adjoint elements $a, b\in A_{s.a.}$ with $\|a\|, \|b\|\le 1,$
if $\|a-b\|<\dt_1,$ then 
\beq\label{Ctheta-3}
\|\theta_{j,i, \eta}(a)-\theta_{j,i,\eta}(b)\|<\eta/2(n+1),\,\,\, 1\le j\le m,
\eneq
where $\theta_{j,i,\eta}$ is defined in \ref{Dpseudosp}. 

Let ${\cal G}=\{e_i: 1\le i\le n\}\cup\{1\}$ and 
${\cal H}=\{\theta_{j,i,\eta}: 1\le i\le m, \,\, 1\le j\le n\}.$ 
For $X=\I^n,$ choose $\dt_2(n, \eta/4)>0$ such that (the second part of) Lemma \ref{Lpurb} holds for 
${\cal H},$ ${\cal G}$ above and 
$\eta/4(n+1)$ (in place of $\eta$).
Choose $\dt_3>0$ such that, for any unital c.p.c. map
$\Phi: C(Y)\to A$ (any unital \CA\, $A$),  that 
\beq
\|\Phi(e_ie_j|_Y)-\Phi(e_i|_Y)\Phi(e_j|_Y)\|<\dt_3\,\,\, (1\le i, j\le n)
\eneq
implies, for $1\le j\le m,$ 
\beq\label{Ctheta-4}
\|\Phi(\Theta_{x_j,\eta}|_Y)-\Phi(\theta_{j,1, \eta}|_Y)\Phi(\theta_{j,2, \eta}|_Y)\cdots \Phi(\theta_{j,n, \eta}|_Y)\|
<\eta/4.
\eneq
Choose $\dt=\min\{\dt_1, \dt_2, \dt_3\}.$ 

Suppose that $L$ and $\{a_i: 1\le i\le n\}$ satisfy the assumption of this lemma for 
$\dt.$ 
Since $\Theta_{x_j,\eta}=\theta_{j,1,\eta}\theta_{j,2,\eta}\cdots \theta_{j,n,\eta}$
and 
\beq\label{Ctheta-5-}
\Theta_{x_j,\eta}(b_1, b_2, ...,b_n)=\theta_{j,1,\eta}(b_1)\theta_{j,2,\eta}(b_2)\cdots \theta_{j,n,\eta}(b_n)
\eneq
for any $n$-tuple $(b_1,b_2,...,b_n)$ of self-adjoint elements in $A$ with $\|b_i\|\le 1$ ($1\le i\le n$),
by the choice of $\dt$ (and $\dt_1$) and  by the second part of \eqref{Ctheta-1} and \eqref{Ctheta-3}, we have,
for each $1\le j\le m,$ 
\beq\label{Ctheta-6}
\|\Theta_{x_j,\eta}(L(e_1|_X), L(e_2|_X), ..., L(e_n|_X))-\Theta_{x_j,\eta}(a_1, a_2,...,a_n)\|<\eta/2.
\eneq
By the choice of $\dt_3,$ 
we also have 
\beq\label{Ctheta-7}
\|L(\Theta_{x_j,\eta}|_X)-L(\theta_{j,1, \eta}(e_1|_X))L(\theta_{j,2, \eta}(e_2|_X))\cdots 
L(\theta_{j,n, \eta}(e_n|_X))\|
<\eta/4.
\eneq
Define $\tilde L: C(\I^n)\to A$ by 
$\tilde L(f)=L(f|_Y)$ for all $f\in C(Y).$ 
By the choice $\dt$ and applying (the second part of) Lemma  \ref{Lpurb} to 
$\tilde L_i=\tilde L,$ $i=1,2,$ 
we obtain
\beq
\|\tilde L(\theta_{j,i,\eta}(e_i))-\theta_{j,i,\eta}({\tilde L}(e_i))\|<\eta/4(n+1),\,\,\,1\le i\le n, \,\,\, 1\le j\le m.
\eneq
In other words,
\beq
\|L(\theta_{j,i,\eta}(e_i)|_Y)-\theta_{j,i,\eta}(L(e_i|_Y))\|<\eta/4(n+1),\,\,\,1\le i\le n, \,\,\, 1\le j\le m.
\eneq
Note that $\theta_{j,i,\eta}(e_i)|_Y=\theta_{j,i,\eta}(e_i|_Y).$
It follows that  (keeping in mind of \eqref{Ctheta-5-})
\beq\nonumber
&&\hspace{-1in}L(\theta_{j,1, \eta}(e_1|_Y))L(\theta_{j,2, \eta}(e_2|_Y))\cdots 
L(\theta_{j,n, \eta}(e_n|_Y))\\\label{Ctheta-8}
&&\hspace{1.3in}\approx_{n\eta\over{4(n+1)}}\Theta_{x_j,\eta}(L(e_1|_Y), L(e_2|_Y), ..., L(e_n|_Y)).
\eneq
Combining \eqref{Ctheta-8}  with \eqref{Ctheta-7}, we obtain
\beq\label{Ctheta-9}
\|L(\Theta_{x_j, \eta}|_Y)-\Theta_{x_j,\eta}(L(e_1|_Y), L(e_2|_Y), ..., L(e_n|_Y))\|<\eta/2,\,\,\, 1\le j\le n.
\eneq
We conclude that, by \eqref{Ctheta-9} and 
\eqref{Ctheta-6},
\beq
&&\hspace{-0.8in}\|L(\Theta_{x_j, \eta}|_Y)-\Theta_{x_j,\eta}(a_1, a_2,...,a_n)\|
\le \|L(\Theta_{x_j, \eta}|_Y)-\Theta_{x_j,\eta}(L(e_1|_Y), L(e_2|_Y), L(e_n|_Y))\|\\
&&\hspace{0.8in}+
\|\Theta_{j,\eta}(L(e_1|_Y), L(e_2|_Y), L(e_n|_Y))-\Theta_{x_j,\eta}(a_1, a_2,...,a_n)\|\\
&&\hspace{1.4in}<\eta/2+\eta/2<\eta.
\eneq

\end{proof}

\begin{df}\label{DHd}
Let $M$ be a metric space (we only consider the case that $M\subset \R^n$ is compact 
such as the case that $M=\I^n$).
Recall that the Hausdorff distance of two compact subsets of $X, Y\subset M$ 
is defined by
\beq
d_H(X, Y)=\max\{\sup_{x\in X} \{{\rm dist}(x, Y)\}, \sup_{y\in Y}\{{\rm dist}(y, X)\}\}.
\eneq
Let $F(M)$ be the set of all non-empty compact subsets of $M.$ 
Then $(F(M), d_H)$ is a compact metric space with the metric $d_H$ when $M$ is compact.
\end{df}

\begin{prop}\label{Pspuniq}
Fix $k\in \N.$  For any $\eta>0,$ there exits $\dt(k, \eta)>0$ satisfying the following:

(1) Suppose that $A$ is a unital \CA\, and $a_1, a_2,...,a_k\in A_{s.a.}$ with $\|a_i\|\le 1$ 
($1\le i\le k$)
such that $(a_1, a_2,...,a_k)$ has 
a non-empty $\dt$-near-spectrum $X=nSp^\dt((a_1,a_2,...,a_n)).$ 
If 
$Y$ is also a non-empty  $\dt$-near-spectrum of $(a_1,a_2,...,a_k),$  then
$Z=s{\rm Sp}^\eta((a_1,a_2,...,a_k))\not=\emptyset,$ 
and 
\beq
d_H(X, Y)<\eta\andeqn  X, Y\subset Z\subset X_{2\eta}.
\eneq

(2) Suppose,  in  addition to (1), $A$ has a (closed two-sided) ideal $J\subset A,$
$\pi: A\to A/J$ is the quotient map, 
and $\Omega_1=s{\rm Sp}^\eta((\pi(a_1), \pi(a_2),...,\pi(a_n))$ and $\Omega_2=nSp^\dt((\pi(a_1), \pi(a_2),...,\pi(a_n))\not=\emptyset.$
Then
\beq
\Omega_1\subset Z,\,\,\, \Omega_2\subset \Omega_1\subset (\Omega_2)_{2\eta}\andeqn \Omega_2\subset \overline{X_\eta}.
\eneq

\end{prop}

\begin{proof}
We may assume that $0<\eta<1.$
We first prove (1).

Let $\{x_1, x_2,..., x_m\}$  be an $\eta/8$-dense subset of $\I^n.$ 
Choose $f_j\in C(\I^n)_+$ with $0\le f_j\le 1,$
$f_j(x)=1$ if ${\rm dist}(x, x_j)\le \eta/4,$ and $f_j(x)=0,$ if 
${\rm dist}(x, x_j)\ge \eta/2,$ $j=1,2,...,m.$
Put ${\cal F}=\{f_j: 1\le j\le m\}.$ 

Note that $\{e_j: 0\le j\le n\}$ is a generating set of $C(\I^n).$
Let  $\dt_1:=\dt({\cal F}, \eta/32)>0$ be  given  by Lemma \ref{Lpurb} (for 
$\eta/32$ and ${\cal F}$ as well as $\I^n$) and 
$\dt_2:=\dt(n,\eta/2)>0$  be given by Corollary \ref{Ctheta}. 
Choose $\dt=\min\{\dt_1/2, \dt_2/2, \eta/33\}.$

Now suppose that $X$ and $Y$ are $\dt$-near-spectra of the $n$-tuple
$(a_1,a_2,...,a_n).$ 
Let $L_1,\, L_2: C(X)\to A$  be  unital c.p.c maps such that
\beq
&&\|L_1(e_ie_j|_X)-L_1(e_i|_X)L_1(e_j|_X)\|<\dt,\,\,\,1\le i,\, j\le n, \,\,\, \\\label{XY-7}
&&\|L_2(e_ie_j|_Y)-L_2(e_i|_Y)L_2(e_j|_Y)\|<\dt,\,\,\,1\le i,\, j\le n, \,\,\, \\\label{XY-7-1}
&&\|L_1(e_j|_X)-a_j\|<\dt, \,\,\, \|L_2(e_j|_Y)-a_j\|<\dt, \,\,\, 1\le j\le n,\andeqn\\\label{XY-8}
&&\|L_1(f|_X)\|\ge (1-\dt) \andeqn \|L_2(f|_Y)\|\ge (1-\dt)
\eneq 
for any $f\in C(\I^n)_+$ with $f(x)=1$ for all $x\in B(\xi, \dt),$ whenever $\xi\in X$ (for  $L_1$),
or $\xi\in Y$ (for $L_2$). 

Let $\xi\in X.$  We may assume that $\xi\in B(x_j, \eta/8).$
Then $B(x_j, \dt)\subset B(\xi, \eta/4).$
By \eqref{XY-8},
\beq\label{Pspuniq-10}
\|L_1(f_j|_X)\|\ge (1-\dt).
\eneq
Define  $L'_l: C(\I^n)\to A$ ($l=1,2$) by 
$L'_1(f)=L_1(f|_X)$ and $L_2'(f)=L_2(f|_Y)$ for all $f\in C(\I^n).$
By \eqref{XY-7} and applying   Lemma \ref{Lpurb}, we have, for $1\le j\le m,$ 
\beq
\|L_1'(f_j)-L_2'(f_j)\|<\eta/32,\,\,\, {\rm or}\,\,\,
\|L_1(f_j|_X)-L_2(f_j|_Y)\|<\eta/32.
\eneq
Then,  by \eqref{Pspuniq-10},  we obtain, for $1\le j\le m,$ 
\beq
\|L_2(f_j|_Y)\|\ge \|L_1(f_j|_X)\|-\eta/32\ge 1-\dt-\eta/32>3/4.
\eneq
Since $f_j(x)=0,$ if ${\rm dist}(x, x_j)\ge \eta/2,$  we must have  ${\rm dist}(x_j, Y)<\eta/2.$
Therefore 
\beq
{\rm dist}(\xi, Y)\le {\rm dist}(\xi, x_j)+{\rm dist}(x_j, Y)<\eta/8+\eta/2=5\eta/8.
\eneq
This  holds for all $\xi\in X.$ 
Exchanging the role of $X$ and $Y,$ we obtain
\beq
d_H(X, Y)<\eta.
\eneq

Let $D^\eta=\{\xi_1,\xi_2,...,\xi_N\}$ and 
 $Z=\bigcup_{\|\Theta_{\xi_j, \eta}(a_1,a_2,...,a_k)\|\ge 1-\eta} \overline{B(\xi_j, \eta)}.$ 

By the choice of  $\dt$ and applying  Corollary \ref{Ctheta}, we have that 
\beq\label{XY-15-}
\|L_1(\Theta_{\xi_j,\eta}|_X)-\Theta_{\xi_j,\eta}(a_1,a_2,...,a_k)\|<\eta/2,\,\,\, j=1,2,...,N.
\eneq
Pick $\xi\in X.$ There is $\xi_j\in D^\eta$ such that
${\rm dist}(\xi, \xi_j)<\eta/2.$  Recall that $\Theta_{\xi_j,\eta}(y)=1$
if ${\rm dist}(y, \xi_j)\le 3\eta/4.$  As $\dt<\eta/32,$  we have $\eta/2+\dt<3\eta/4.$ Hence 
 $\Theta_{\xi_j, \eta}(y)=1$ for all $y\in B(\xi, \dt)\subset B(\xi_j, 3\eta/4).$ 
It follows that 
\beq
\|L_1(\Theta_{\xi_j,\eta}|_X)\|\ge 1-\dt.
\eneq
Thus,  by \eqref{XY-15-},
\beq
\|\Theta_{\xi_j,\eta}(a_1,a_2,...,a_k)\|\ge 1-\dt-\eta/2\ge 1-\eta.
\eneq
This implies that 
\beq
X\subset \bigcup_{\|\Theta_{\xi_j,\eta}(a_1,a_2,...,a_k)\|\ge 1-\eta}\overline{B(\xi_j, \eta)}=s{\rm Sp}^\eta((a_1,a_2,...,a_k)).
\eneq
Since $X\not=\emptyset,$ we conclude that $Z\not=\emptyset.$
The same argument shows that $Y\subset s{\rm Sp}^\eta((a_1,a_2,...,a_k)).$ 

Next, let $\|\Theta_{\xi_j,\eta}(a_1,a_2,...,a_k)\|\ge 1-\eta.$ 
By \eqref{XY-15-},
\beq
\|L_1(\Theta_{\xi_j, \eta}|_X)\|\ge 1-\eta-\eta/2>1/4.
\eneq
Hence $\Theta_{\xi_j, \eta}|_X\not=0.$ 
Thus 
\beq
X\cap B(\xi_j, \eta)\not=\emptyset. 
\eneq
It follows that
\beq
B(\xi_j, \eta)\subset X_{2\eta}. 
\eneq
Hence 
\beq
Z\subset X_{2\eta}.
\eneq
This completes the proof of part (1). 

For (2),  let us keep notation above.  Then $\Omega_1\subset Z$ 
is immediate. 

Suppose that $\Omega_2\not=\emptyset.$ 
%
Then  that $\Omega_2\subset \Omega_1
\subset (\Omega_2)_{2\eta}$ follow  from part (1). 
It remains to show that $\Omega_2\subset \overline{X_\eta}.$


There is a unital c.p.c. map
$\Phi: C(\Omega_2)\to A/J$ such that
\beq
&&\|\Phi(e_ie_j|_{\Omega_2})-\Phi(e_i|_{\Omega_2})\Phi(e_j|_{\Omega_2})\|<\dt,\,\,\, 1\le i,j\le k;\\\label{XY-30}
&&\|\Phi(e_j|_{\Omega_2})-\pi(a_j)\|<\dt,\,\,\, 1\le j\le k,\andeqn\\\label{XY-20}
&&\|\Phi(f)\|\ge 1-\dt
\eneq
for any $f\in C(\Omega_2)_+$ with $f(x)=1$ if $x\in B(\zeta, \dt)$ for some $\zeta\in \Omega_2.$ 

Choose $y\in \Omega_2.$ 
We may assume that ${\rm dist}(y, x_j)<\eta/8.$  
Hence $B(y, \dt)\subset {\rm dist}(x_j,\eta/4).$ 
It follows from \eqref{XY-20} that
\beq
\|\Phi(f_j)\|\ge 1-\dt.
\eneq
Let $\tilde L_1: C(\I^n)\to A/J$ be defined $\tilde L_1(f)=\pi\circ L_1(f|_X)$ 
and $\tilde \Phi: C(\I^n)\to A/J$ by $\tilde \Phi(f)=\Theta(f|_{\Omega_2})$
for all $f\in C(\I^n).$
By the choice of $\dt,$  \eqref{XY-7}, \eqref{XY-30} and \eqref{XY-7-1}, and applying Lemma \ref{Lpurb}
to $\pi\circ \tilde L_1$ and $\tilde \Phi,$  we have  
\beq
\|\pi\circ \tilde L_1(f_j)-\tilde \Phi(f_j)\|<\eta/32,\,\,\, 1\le j\le m.
\eneq
In other words,
\beq
\|\pi\circ L_1(f_j|_X)-\Phi(f_j|_{\Omega_2})\|<\eta/32, \,\,\, 1\le j\le m.
\eneq
Thus
\beq
\|L_1(f_j|_X)\|\ge \|\pi\circ L_1(f_j|_X)\|\ge 1-\dt-\eta/32\ge 3/4.
\eneq
Hence $f_j|_X\not=0.$ It follows that ${\rm dist}(x_j, X)<\eta/2.$
\beq
{\rm dist}(y, X)<\eta/2+\eta/8<\eta.
\eneq
Hence $\Omega_2\subset \overline{X_\eta}.$
\end{proof}

\begin{prop}\label{Pesspsp}
Let $\eta>0.$ There exists $\dt(\eta)>0$ satisfying the following:
Suppose that $A$ is a unital \CA\, and 
$a_1, a_2\in A_{s.a.}$ such that $\|a_i\|\le 1$ ($i=1,2$), and 
\beq\label{Pesspsp-1}
\|a_1a_2-a_2a_1\|<\dt.
\eneq
Then  ${\rm sp}(a_1+ia_2)\subset s{\rm Sp}^\eta((a_1,a_2)).$
\end{prop}

\begin{proof}
Otherwise, one obtains 
$\eta_0>0,$ 
a sequence of unital \CA s $\{A_n\},$ a sequence of elements 
$a_n, b_n\in (A_n)_{s.a.}$ 
with $\|a_n\|, \|b_n\|\le 1$
such that
\beq\label{esspsp-2}
\lim_{n\to\infty}\|a_nb_n-b_na_n\|=0,
\eneq
and 
 a sequence of numbers $\lambda_n\in {\rm sp}(a_n+ib_n)
 \setminus s{\rm Sp}^{\eta_0}((a_n, b_n)).$
 Put $c_n=a_n+ib_n$ and  $Z_n=s{\rm Sp}^{\eta_0/4}((a_n, b_n)),$ $n\in \N.$
 Note that $\lambda_n\in \I^2\subset \C$  and $Z_n\subset \I^2,$ $n\in \N.$
 Since $(F(\I^2), d_H)$ is compact, without loss of generality,  by passing to a subsequence,
 we may assume that $Z_n\to Z$ for some compact subset $Z\subset \I^n$
 in $(F(\I^2), d_H)$ as $n\to\infty,$ and 
 $\lambda_n\to \lambda\in \I^2.$ 
 
 By \eqref{pSpdteta}, $(s{\rm Sp}^{\eta_0/4}((a_n, b_n)))_{\eta_0/4}\subset s{\rm Sp}^{\eta_0}((a_n, b_n)),$
 $n\in \N.$  It follows that 
 \beq\label{esspsp-5}
 {\rm dist}(\lambda_n, Z_n)\ge \eta_0/4,\,\,\,n\in \N.
 \eneq
 Hence, there exists $n_0\in \N$ 
 such that
 \beq\label{esspsp-6}
 \dist(\lambda, Z_{n})\ge \eta_0/5 \rforal n\ge n_0.
 \eneq
 
Let  $B=\prod_nA_n$ and $\Pi: B\to B/\bigoplus_nA_n$  be the quotient map. 
Put $x=\Pi(\{c_n\}).$ Then, by \eqref{esspsp-2}, $x$ is normal.
Let ${\rm sp}(x)=X.$ 
Then $\phi: C(X)\to B/\bigoplus_nA_n$ defined by $\phi(f)=f(x)$
for $f\in C(X)$ is a unital monomorphism. 
By the Choi-Effros Lifting Theorem (\cite{CE}), there is a sequence 
of unital c.p.c. maps $L_n: C(X)\to A_n$
such that $\Pi\circ \{L_n\}=\phi.$ 
Let $\dt:=\dt(2, \eta_0/4)$ be as in Proposition \ref{Pspuniq}. Then, 
since $\phi$ is a \hm, for all large $n,$
$X$ is a (non-empty) $\dt$-near-spectrum for $(a_n, b_n).$ 
Applying Proposition \ref{Pspuniq}, we may assume that, for all $n,$ 
\beq\label{esspsp-7}
X\subset Z_n.
\eneq
Note that, by \eqref{esspsp-6} and \eqref{esspsp-7},  $\lambda\not\in X.$
Let $y=(\lambda-x)^{-1}.$ 
Note that $\Pi(\{\lambda_n-(a_n+ib_n)\})=\lambda-x.$
Hence there are $y_n\in A_n$ such that ($\|y_n\|\le \|y\|+1$)
\beq
\Pi(\{y_n\})=y\andeqn \lim_{n\to\infty}\|y_n(\lambda_n-(a_n+ib_n))-1\|=0.
\eneq
Hence, for all large $n,$ 
\beq
\lambda_n\not\in {\rm sp}(a_n+ib_n).
\eneq
This is a contradiction. 
So the lemma follows.
%
\end{proof}

\section{Purely infinite simple \CA s}

Recall that the Calkin algebra $B(H)/{\cal K}$ (when $H$ is infinite dimensional separable Hilbert space)
is purely infinite and simple. The results in the section will be used in the next four sections.

\begin{df}\label{Dbrick}
An even $1/k$-
brick $\mathtt{b}=\mathtt{b}^k$  with corner $\xi=(x_1, x_2,...,x_n)\in P_k$ (with $M=1$ -- see  Definition \ref{n-net})  is a subset of $\I^n$ of the form
\beq
\{(y_1,y_2,...,y_n): x_i\le y_i\le x_i+1/k,\, \,x_i, x_i+1/k\in P_k\}.
\eneq
For the remaining of this paper, {\em all $1/k$-bricks are assumed to be even and have
 their corner $\xi\in P_k.$} An open $1/k$-brick $\mathtt{b}^o$ is the interior of the $1/k$-brick
 $\mathtt{b}.$  The  boundary $\partial(\mathtt{b})$ of a $1/k$-brick $\mathtt{b}$ is defined as 
 $\mathtt{b}\setminus \mathtt{b}^o.$
 
{\it A $1/k$-brick combination} $\mathtt{B}^{k}$ is a finite union of $1/k$-bricks. 
Since two distinct  $1/k$ - bricks have distinct corners, 
a $1/k$-brick combination $\mathtt{B}^k$ has the following easy property:
If $\mathtt{b}_1, \mathtt{b}_2$ are two distinct  $1/k$-bricks of $\mathtt {B}^k,$ 
then $\mathtt{b}_1^o\cap \mathtt{b}_2^o=\emptyset.$


\end{df}

We believe  that
Proposition \ref{PKfinite} to be  well known.

\begin{prop}\label{PKfinite}
If $Y$ is a union of finitely many compact convex subsets  in $\R^n$ (for some integer $n\in \N$),  then 
$K_i(C(Y))$ is finitely generated, $i=0,1.$
\end{prop}

\begin{proof}

We prove this by induction on the number of convex sets.

If $Y$ is convex, then $Y$ is contractive. Hence $K_0(C(Y))=\Z$ and 
$K_1(C(Y))=\{0\}.$

Suppose that the proposition  is proved for $m$ many closed convex sets, $m\ge 1.$

Now suppose that $Y$ is a union of $m+1$ many compact convex sets in $\R^n.$
Write $Y=\cup_{i=1}^{m+1} Y_i,$
where each $Y_i$ is a compact convex subset of $\R^n.$ 

Put $X_1=\cup_{i=1}^m Y_i.$ Then 
$Y_{m+1}\cap X_1=\cup_{i=1}^m (Y_i\cap Y_{m+1})$ 
which is a union of $m$ many compact convex subsets. 
By the induction assumption, $K_i(C(X_1\cap Y_{m+1}))$ is finitely generated ($i=0,1$). 

Then 
\beq
C(Y)=\{(f, g): (f, g)\in C(X_1)\oplus C(Y_{m+1}): f|_{X_1\cap Y_{m+1}}=g|_{X_1\cap Y_{m+1}}\}
\eneq
is a pull back. By the Mayer-Vietoris  sequence in $K$-theory for \CA s (see \cite{Hj86}),
we obtain the following commutative diagram. 

\beq
\begin{array}{ccccc}
{\small{K_0(C(Y))}} &
{\longrightarrow}&  K_0(C(X_1))\oplus K_0(C(Y_{m+1})) &
{\rightarrow}
& K_0(C(X_1\cap Y_{m+1}))\\
\uparrow
& && & \downarrow\\
K_1(C(X_1\cap Y_{m+1})) &
{\leftarrow}&  K_i(C(X_1))\oplus K_1(C(Y_{m+1})) &
{\leftarrow}
& K_i(C(Y))\,.
\end{array}
\eneq

Since we have established that $K_i(C(X_1)),$ $K_i(C(Y_{m+1}))$ and 
$K_i(C(X_1\cap Y_{m+1}))$ are all finitely generated abelian groups, we conclude that
$K_i(C(Y))$ is also finitely generated.  This ends the induction.

\end{proof}

\begin{df}\label{Dbrickcover}
Let $X$ be a compact subset of $\I^n.$    Define  $\mathtt{B}_X^k$ to be the union of all 
even $1/k$-bricks  (in $\I^n$ with corners at $P_k$ -- see  Definition \ref{n-net}) which has a non-empty intersection 
with $X.$

We will use the following easy facts:


(i) $X\subset \mathtt{B}_X^k;$
(ii) If $\mathtt{b}^k$ is one of  the bricks in $\mathtt{B}_X^k,$ then $\mathtt{b}^k\cap X\not=\emptyset;$
(iii) For any $\zeta\in \mathtt{B}_X^k,$ ${\rm dist}(\zeta, X)\le \sqrt{n}/k;$ and 
(iv) $K_i(C(\mathtt{B}_X^k)$ is finitely generated (this follows from Proposition \ref{PKfinite}).
%

%
In what follows $\mathtt{B}_X^k$ is called the $1/k$-brick cover of $X.$
(The introduction of even $1/k$-bricks is for convenience not essential.)
\end{df}

\begin{df}\label{DN}
Let $X$ be a compact metric space. We say that $X$ has property (F), if 
$X=\sqcup_{i=1}^m X_i$  is a disjoint union of finitely many path connected compact metric spaces
$X_i$ such that  $K_i(C(X_j))$ is finitely generated ($i=0,1$), $j=1,2,...,m.$
Choose a base point $x_i\in X_i.$ Write 
$Y_i=X_i\setminus \{x_i\},$ $i=1,2,...,m.$
Denote by $\Omega(X)=\sqcup_{i=1}^m Y_i.$ 

Let $B$ be a \CA\, and $\phi: C(X)\to B$ 
a \hm. Denote by $\phi^\omega=\phi|_{C_0(\Omega(X))}: C_0(\Omega(X))\to B$ the restriction.  

Every finite CW complex $X$ has property (F). 
If $X$ is a $1/k$-brick combination in $\I^n,$  then (by Proposition \ref{PKfinite}) $X$ also has property (F). 
%
\end{df}

Recall that if $\psi: C(X)\to B$ has finite dimensional range, then there are 
mutually orthogonal projections $p_1, p_2, ...,p_l\in B$ and 
$x_1, x_2,...,x_l\in X$ such that
\beq
\psi(f)=\sum_{i=1}^l f(x_i)p_i\rforal f\in C(X).
\eneq

\begin{thm}\label{T-1fs}
Let $X$ be a compact metric space  with property (F).
Suppose that $B$ is a unital purely infinite simple \CA\, and $\phi: C(X)\to B$
is a unital  injective \hm\, such that 
\beq\label{T-1fs-1}
[\phi^\omega]=0\,\,\,{\rm in}\,\,\, KK(C_0(\Omega(X)), B)
\eneq
(see Definition \ref{DN} for $\phi^\omega$).
Then, for any $\ep>0$ and any finite subset ${\cal F},$ there exists 
a set of finitely many mutually orthogonal projections $p_1,p_2,...,p_k$
with $\sum_{i=1}^kp_i=1$ and 
$x_1, x_2,...,x_k\in X$ such that 
\beq\label{T-1fs-2}
\|\psi(f)-\phi(f)\|<\ep \tforal f\in {\cal F},
\eneq
where $\psi(f)=\sum_{j=1}^k f(x_j)p_j$ for all $f\in C(X).$
Moreover,  if $X$ is connected, 
we may require that either $[p_j]=0,$ or $[p_j]=[1_B],$ in $K_0(B),$ $j=1,2,...,k.$
\end{thm}

\begin{proof}
 Since $X$ has property (F), we may write 
 $C(X)=C(\sqcup_{i=1}^m X)=\bigoplus_{i=1}^m C(X_i),$ where each $X_i$ is path connected
 and $K_j(C(X_i))$ is finitely generated ($j=0,1$), $ 1\le i\le m.$
 Let $d_i\in C(X)$ be such that ${d_i}|_{X_i}=1$ and $d_i(x)=0$ if $x\not\in X_i,$ $1\le i\le m.$
 Put $E_i=\phi(d_i)$ ($1\le i\le m$). By considering $\phi|_{C(X_i)}\to E_iBE_i$ for each $i,$
 we reduce the 
general case to the case that $X$ has only one path connected component.
Note that, if originally, $X$ is not connected, then we do not need to consider the ``Moreover" part 
of the theorem. If $X$ is connected, then we need to point out that $\phi(1_{C(X)})=1_A.$

So now we assume that $X$ is connected. 
Put $\Omega:=\Omega(X)=X\setminus \{x\}$ for some 
$x\in X.$ 
Then we have a splitting short exact sequence 
$$
0\to C_0(\Omega)\to C(X)\to \C\to 0.
$$
Recall that $K_i(C(X))$ is finitely generated ($i=0,1$).
It follows  that, for any unital \CA\, $B,$
\beq
KK(C(X), B)=KK(C_0(\Omega), B)\oplus KK(\C, B).
\eneq
Define $\phi_1: C(X)\to B$ by $\phi_1(f)=f(x)\cdot 1_B$ for all $f\in C(X).$
Then one computes that, by \eqref{T-1fs-1},
\beq
[\phi_1]=[\phi]\,\,\,{\rm in}\,\, \, KK(C(X), B).
\eneq
Fix $\ep>0$ and finite subset ${\cal F}\subset C(X).$
Applying Theorem A of \cite{Dd}, 
we obtain an integer $k,$  a unitary $u\in M_{k+1}(B)$
and $x_1, x_2,...,x_k\in X$ 
such that, for all $f\in {\cal F},$
\beq\label{T1-9}
\|\diag(\phi(f), f(x_1), f(x_2),...,f(x_k))-u^*\diag(\phi_1(f), f(x_1),f(x_2),...,f(x_k))u\|<\ep/3.
\eneq

Define $\phi_0: C(X)\to M_k(B)$ by
$\phi_0(f)=\diag(f(x_1), f(x_2),...,f(x_k))$ for all $f\in C(X).$
Then we may rewrite $\phi_0(f)=\sum_{i=1}^k f(x_i)p_i$
for all $f\in C(X),$ where $\{p_1, p_2,...,p_k\}$ is a set of mutually orthogonal 
projections such that $\sum_{i=1}^k p_i=1_{M_k}.$ 
Choose mutually orthogonal  projections $p_{0,1}, p_{0,2},...,p_{0,k}\in M_{k+1}(B)$ such that 
$p_01_{M_k}=1_{M_k} p_0=0,$ where $p_0=\sum_{i=1}^k p_{0,i},$ and 
$[p_{0,i}+p_i]=0$ in $K_0(B)$ (recall that $B$ is purely infinite simple).
Define $\phi_{00}: C(X)\to (1_{M_k}+p_0)M_{k+1}(B)(1_{M_k}+p_0)\subset M_{k+1}(B)$
by $\phi_{00}(f)=\sum_{i=1}^k f(x_i)(p_i+p_{0,i})$ for all $f\in C(X).$
By replacing $\phi_0$ by $\phi_{00}$ and $u$ by
$u\oplus 1_B,$ we may assume, without loss of generality, 
 that $[p_i]=0,$ $i=1,2,...,k,$ and 
\beq\label{T1-10}
\|\diag(\phi(f), \phi_0(f))-u^*\diag(\phi_1(f), \phi_0(f))u\|<\ep/3\rforal f\in {\cal F}.
\eneq
(But then $u$ is a unitary in $M_{k+2}(B)$).

Recall that, by  S. Zhang's theorem (\cite{Zh}), $B$ has real rank zero.
Since $\phi$ is injective, by  applying Lemma 4.1 of \cite{EGLP}, we obtain a set of mutually orthogonal  non-zero 
projections $q_1, q_2, ...,q_k\in B$ with $q=\sum_{i=1}^kq_i<1_B$ 
such that
\beq
\|\phi(f)-((1-q)\phi(f)(1-q)+\sum_{i=1}^k f(x_i) q_i)\|<\ep/3\rforal f\in {\cal F}.
\eneq
There are $q_i'\le q_i$ such that  $[q_i']=0$ in $K_0(B),$ $i=1,2,...,k$
(recall again that $B$ is purely infinite  simple). 
We may then write 
\beq\label{T1-11}
\|\phi(f)-(\gamma(f)+\sum_{i=1}^kf(x_i)q_i')\|<\ep/3\rforal f\in {\cal F},
\eneq
where  $\gamma(f)=(1-q)\phi(f)(1-q)+\sum_{i=1}^k f(x_i)(q_i-q_i')$
for all $f\in C(X).$
There 
are $w_i\in M_{2k+3}(B)$
such that
\beq
w_i^*w_i=q_i'\andeqn w_iw_i^*=q_i'\oplus p_i,\,\,\, i=1,2,...,k.
\eneq
Let $w=(1-q)+\sum_{i=1}^k (q_i-q_i')+\sum_{i=1}^k w_i.$
Then $w^*w=1_B$ and $ww^*=1_{M_{k+1}}+\sum_{i=1}^k p_i.$ Moreover, 
\beq
w^*(\diag(\gamma(f)+\sum_{i=1}^k f(x_i)q_i', \phi_0(f))w=\gamma(f)+\sum_{i=1}^kf(x_i)q_i'
\eneq
Hence
\beq\label{T1-13}
\|\phi(f)-w^*(\diag(\gamma(f)+\sum_{i=1}^k f(x_i)q_i', \phi_0(f))w\|<\ep/3\rforal f\in {\cal F}.
\eneq
Define, for all 
$f\in C(X),$ 
\beq\label{T1-20}
\psi(f)=w^*u^*(\diag(\phi_1(f), \phi_0(f))uw=
f(x) w^*u^*1_Au^*w+\sum_{i=1}^k f(x_i)w^*u^*p_iuw
\eneq
Then $\psi$ is a unital \hm\, from $C(X)$ to $B$  with finite dimensional range
and, by  \eqref{T1-13}, \eqref{T1-11} and \eqref{T1-10},  for all $f\in {\cal F},$
\beq\nonumber
\hspace{-0.15in}\|\phi(f)-\psi(f)\| &\le& \|\phi(f)-w^*(\diag(\gamma(f)+\sum_{i=1}^k f(x_i)q_i', \phi_0(f)))w\|\\\nonumber
&& \hspace{0.2in}+\|w^*(\diag(\gamma(f)+\sum_{i=1}^k f(x_i)q_i', \phi_0(f)))w-w^*u^*\diag(\phi_1(f), \phi_0(f))uw\|\\\nonumber
 &<&\ep/3+\|\diag(\gamma(f)+\sum_{i=1}^k f(x_i)q_i', \phi_0(f))-u^*\diag(\phi_1(f), \phi_0(f))u\|\\\nonumber
 &<&\ep/3+\ep/3+\|\diag(\phi(f), \phi_0(f))-u^*\diag(\phi_1(f), \phi_0(f))u\|<2\ep/3+\ep/3=\ep.
\eneq
Finally, since $[p_i]=0$ in $K_0(B)$ as we arranged earlier,  by  \eqref{T1-20},
the  ``Moreover" part of the statement also holds.
\end{proof}

Let $X\subset \I^n$ be a compact subset. We use the usual Euclidian metric.
Recall that, for any $\eta>0,$ 
\beq
X_\eta=\{\xi\in \I^n: {\rm dist}(\xi, X)< \eta\}.
\eneq

\begin{thm}\label{C1}
Fix $n\in \N.$ For any $\ep>0$ and finite subset ${\cal F}\subset C(\I^n),$ 
there exists $\dt(n, \ep)>0$ satisfying the following:

Let $X$ be a compact subset of $\I^n.$ 
%
Suppose that $B$ is a unital purely infinite simple \CA\, and 
$\phi: C(X)\to B$ is a unital injective \hm. 
If $Y$ is any compact subset with property (F) such that
$X \subset Y\subset X_\dt,$ 
then there is  
a unital injective \hm\, $\psi: C(Y)\to B$ such that
\beq\label{C1-f1}
\|\phi(f|_X)-\psi(f|_Y)\|<\ep\tforal f\in {\cal F},
\,\,\, [\psi]=[\phi\circ \pi_X]\,\,\,in\,\,\, KK(C(Y), B),
\eneq
where  $\pi_X: C(Y)\to C(X)$ is the quotient map by restriction ($\pi_X(g)=g|_X$ for all $g\in C(Y)$). 
\end{thm}

\begin{proof}
%
Let $\ep>0$ and finite subset ${\cal F}\subset C(\I^n)$ be given.
Choose $0<\dt<\min\{\ep/2, 1/8\}$ such that, for all $g\in {\cal F},$ 
\beq\label{C1-09}
|g(\xi)-g(\xi')|<\ep/4\rforal \xi, \xi'\in \I^n\,\,\,{\rm with}\,\,\, {\rm dist}(\xi, \xi')<2\dt.
\eneq

Suppose that $Y$ is a 
compact subset of $\I^n$ such that
$X\subset Y\subset X_\dt$ which has property (F).
Since $Y$ is compact, there is $\dt_1>0$ with $0<\dt_1<\dt$ such that
\beq
Y\subset X_{\dt_1}.
\eneq

Put $\eta_0={\dt-\dt_1\over{4}}>0$  and $r=\dt_1+\eta_0.$ Then 
$0<r=\dt_1+\eta_0<\dt_1+2\eta_0<\dt.$
There is a finite $\eta_0$-net $\{\xi_1, \xi_2,..., \xi_m\}\subset X$ such that 
$\cup_{i=1}^m B(\xi_i,\eta_0)\supset X,$
where $B(\xi_i, \eta_0)=\{\xi\in \R^n: {\rm dist}(\xi, \xi_i)< \eta_0\}.$  
Then $\cup_{i=1}^m B(\xi_i, r)\supset Y.$  Note that we assume that $\xi_i\not=\xi_j,$
if $i\not=j.$

Choose a non-zero projection $p\in B$ with $[p]=0$ in $K_0(B)$ and 
$1-p\not=0.$   Since $pBp$ is a purely infinite simple 
\CA, there is a unital embedding
$\Phi: O_2\to pBp.$  Applying Theorem 2.8 of \cite{KP}, one obtains 
a unital embedding 
$\psi_0: C(Y)\to \Phi(O_2)\subset pBp.$ Note  $[\psi_0]=0$ in $KK(C(Y),pBp).$
There is a partial isometry   $w\in B$ such that $w^*w=1-p$ and $ww^*=1.$ 
Define $\psi_1: C(Y)\to B$ by $\psi_1(f)=w^*(\phi(f|_X)w\oplus \psi_0(f)$
for all $f\in C(Y).$ Then $\psi_1$ is a unital injective \hm.

Choose mutually orthogonal non-zero projections $p_1, p_2,...,p_m\in \Phi(O_2)\subset pBp$ 
such that $\sum_{i=1}^m p_i=p.$ Define 
$\phi_0: C(X)\to \Phi(O_2)\subset pBp$ by $\phi_0(f)=\sum_{i=1}^mf(\xi_i)p_i.$
Then $[\phi_0]=0$ in $KK(C(X),B).$
Define $\phi_1: C(X)\to B$ by $\phi_1(f)=w^*\phi(f)w+\phi_0(f)$ for all $f\in C(X).$
  Then $\phi_1$ is injective.
One computes that
\beq
[\phi]=[\phi_1]\,\,\, {\rm in}\,\,\, KK(C(X), B).
\eneq
It follows from Theorem 1.7 of \cite{Dd}  that there is a unitary $u\in B$ such that
\beq\label{C1-10}
\|\phi(f|_X)-u^*\phi_1(f|_X)u\|<\ep/4\rforal f\in {\cal F}.
\eneq

Applying Theorem \ref{T-1fs}, there is a unital \hm\, $\psi_{00}: C(Y)\to \Phi(O_2)\subset pBp$ with finite dimensional 
range such that
\beq\label{C1-11}
\|\psi_0(g|_Y)-\psi_{00}(g|_Y)\|<\ep/4\rforal g\in {\cal F}.
\eneq
We may write $\psi_{00}(g)=\sum_{j=1}^K g(y_j) d_i$
for all $g\in C(Y),$ where $y_j\in Y$ and $d_1, d_2,...,d_K$
are mutually orthogonal non-zero projections  in $\Phi(O_2)$ with $p=\sum_{j=1}^K d_j.$
By choosing a better approximation if necessary, without loss of generality, 
we may assume that $\{y_1, y_2,...,y_K\}$ is $\eta_0$-dense in $Y,$
$K\ge m.$ 
Recall that $\{\xi_1, \xi_2,...,\xi_n\}$ is $\eta_0$-dense in $X$
and $r$-dense in $Y.$
For each $i,$ choose exactly one $y_{(i,1)}\subset \{y_1, y_2,...,y_K\}$
such that ${\rm dist}(y_{(i,1)}, \xi_i)<\eta_0.$
Then, we may assume 
$\cup_{i=1}^m\{y_{(i, j)}: 1\le j\le k(i)\}=\{y_1, y_2,...,y_N\},$
where $\{y_{(i,j)}: 1\le j\le k(i), 1\le i\le m\}$ is a set of distinct points 
such that 
\beq
{\rm dist}(y_{(i,j)}, \xi_i)<\dt_1+2\eta_0<\dt,\,\,\,1\le j\le k(i),\,\,\,1\le i\le n.
\eneq 
Note that we view $(i,j)\in \{1,2,..,N\}.$ Then set 
$q_i'=\sum_{j=1}^{k(i)}d_{(i,j)}.$ 
Put $\psi_{00}': C(Y)\to \Phi(O_2)\subset pBp,$ where 
$\psi_{00}'(g)=\sum_{i=1}^m g(\xi_i)q_i'$ for all $g\in C(Y).$

By the choice of $\dt,$ 
we have, for all $g\in {\cal F},$
\beq\label{C1-12}
\|\psi_{00}(g|_Y)-\psi_{00}'(g|_Y)\|<\ep/4.
\eneq
Recall that $p_i, q_i'\in \Phi(O_2).$ Hence $[p_i]=[q_i']=0$ in $K_0(B).$
There are partial isometries 
$w_i\in pB(u^*pu)$ such that $w_i^*w_i=u^*p_iu$ and $w_iw_i^*=q_i',$ $i=1,2,...,m.$ 
Define $v=(1-p)u\oplus \sum_{i=1}^m w_i.$ 
Then 
\beq\label{C1-13}
v^*v&=&u^*(1-p)u\oplus \sum_{i=1}^m u^*p_iu=u^*(1-p)u+u^*pu=1\andeqn\\
vv^*&=&(1-p)uu^*(1-p)\oplus \sum_{i=1}^m q_i'=(1-p)+\sum_{i=1}^Kd_i=1.
\eneq
So $v$ is a unitary in $B.$
Note that
\beq\label{C1-14}
v^*\psi_{00}'(g)v=v^*(\sum_{i=1}^n g(\xi_i)q_i')v=
\sum_{i=1}^n g(\xi_i) u^*p_iu=u^*\phi_0(g|_X)u \rforal g\in C(Y).
\eneq
Then, combining with \eqref{C1-11} and \eqref{C1-12}, for all $g\in {\cal F},$ 
\beq\nonumber
\|v^*\psi_0(g|_Y)v-u^*\phi_0(g|_X)u\|&=&\|v^*\psi_0(g|_Y)v-v^*\psi_{00}(g|_Y)v\|+
\|v^*\psi_{00}(g|_Y)v-v^*\psi_{00}'(g|_Y)v\|\\\label{C1-15+}
&<&\ep/4+\ep/4=\ep/2.
\eneq
Define $\psi: C(Y)\to B$ by 
\beq\label{C1-15}
\psi(g)=v^*(w^*\phi(g|_X)w+\psi_0(g|_Y))v
\rforal g\in C(Y).
\eneq  Then $\psi$ is a unital injective \hm.
Moreover, for all $g\in {\cal F},$ by \eqref{C1-10}, the definition of $\phi_1$
and $v$  above and by \eqref{C1-15+}, 
\beq\nonumber
\|\phi(g|_X)-\psi(g|_Y)\|&=&\|\phi(g|_X)-u^*\phi_1(g|_X)u\|+\|u^*\phi_1(g|_X)u-\psi(g|_Y)\|\\\nonumber
&<&\ep/4+\|u^*w^*\phi(g|_X)wu\oplus u^*\phi_0(g|_X)u-v^*(w^*\phi(g|_X)w\oplus \psi_0(g|_Y))v\|\\\nonumber
&=&\ep/4+\|u^*w^*\phi(g|_X)wu\oplus u^*\phi_0(g|_X)u-v^*((1-p)w^*\phi(g|_X)w(1-p)\oplus \psi_0(g|_Y))v\|\\\nonumber
&=&\ep/4+\|u^*w^*\phi(g|_X)wu\oplus u^*\phi_0(g|_X)u-u^*w^*\phi(g|_X)wu\oplus v^*\psi_0(g|_Y))v\|\\
&=&\ep/4+\|u^*\phi_0(g|_X)u-v^*\psi_0(g|_Y)v\|<\ep/4+\ep/2<\ep.
\eneq
To see the second formula in \eqref{C1-f1}, 
we may write, from \eqref{C1-15}, since $\psi_0$ factors through $\Phi(O_2),$
\beq
[\psi]=[\phi\circ \pi_X]+[\psi_0]=[\phi\circ \pi_X] \,\,\, {\rm in}\,\,\, KK(C(Y), B).
\eneq
\end{proof}

\begin{cor}\label{C1C}
Fix an integer $n\in \N$ and $1>\ep>0.$

Suppose that $X$ is  a compact subset of\, $\I^n,$ 
$B$ is a unital purely infinite simple \CA\, and 
$\phi: C(X)\to B$ is a unital injective \hm. 
If $Y$ is any compact subset with property (F) such that
$X\subset Y\subset X_\ep,$  then there is  
a unital injective \hm\, $\psi: C(Y)\to B$ such that
\beq\label{C1C-f1}
\|\phi(e_j|_X)-\psi(e_j|_X)\|<\ep\tand
\,\,\, [\psi]=[\phi\circ \pi_X]\,\,\,{\rm in}\,\,\, KK(C(Y), B),
\eneq
where  $\pi_X: C(Y)\to C(X)$ is the quotient map by restriction ($\pi(g)=g|_X$ for all $g\in C(Y)$).
\end{cor}

\begin{proof}
At the beginning of the proof of Theorem \ref{C1},  choose 
${\cal F}=\{e_j: 1\le j\le n\}.$
Then, with $\dt:=\ep/8,$ when ${\rm dist}(\xi, \xi')<2\dt,$
\beq
|e_j(\xi)-e_j(\xi')|<\ep/4.
\eneq
Then the rest of the proof of Theorem \ref{C1} applies (with $\dt=\ep/8$).
\end{proof}

In the proof of Theorem \ref{C1}, we may choose $Y$ to be a $1/k$-brick combination
with $\sqrt{n}/2k<\dt$ (the same $\dt$ in the proof of Theorem \ref{C1}).  Thus we may state a variation of Theorem \ref{C1} as follows.

\begin{cor}\label{Cc1}
Let  $n\in \N.$
Then, for any $\ep>0$ 
and any finite subset ${\cal F}\subset C(\I^n),$ there exists 
$0<\eta<\ep/2$ satisfying the following:
Suppose that  $B$ is a purely infinite simple \CA\, and $\phi: C(X)\to B$ is a unital injective \hm,
where $X\subset \I^n$ is a compact subset.   
For any  integer $k\in \N$ with $\sqrt{n}/k<\eta,$ 
there exists 
a unital injective \hm\, $\psi: C(\mathtt{B}_X^k)\to B$ such that
\beq
\|\phi(f|_X)-\psi(f)\|<\ep\rforal f\in {\cal G}\tand [\psi]=[\pi\circ \phi]\,\,\,{\rm in}\,\,\, KK(C(\mathtt{B}_X^k), B),
\eneq
where  $\pi: C(\mathtt{B}_X^k)\to C(X)$ is the quotient map by restriction, 
${\cal G}\subset C(\mathtt{B}_X^k)$ is a finite subset such that 
${\cal F}\subset \{g|_X: g\in {\cal G}\}$ (and where $\mathtt{B}_X^k$ is a $1/k$-brick cover of $X$).
\end{cor}

\begin{thm}\label{s2-T1}
Let $\ep>0$ and $n\in \N.$  There is $\dt=\dt(n,\ep)>0$ satisfying the following:
If $A$ is a unital purely infinite simple \CA\, and  $s_1, s_2,...,s_n\in A_{s.a.}$ with $\|s_j\|\le 1$ ($1\le j\le n$) 
such that 
\beq
\|s_is_j-s_js_i\|<\dt,\,\,\, i,j=1,2,...,n,
\eneq
then there is integer $k\in \N$  with $2\sqrt{n}/k<\dt$ and $1/k\le \eta<\dt,$ 
and a $1/k$-brick combination $X\subset I^n$ 
and a unital injective \hm\, 
$\phi: C(X)\to A$ such that
\beq
\|\phi(e_j|_X)-s_j\|<\ep,\,\,\, j=1,2,...,n.
\eneq
\end{thm}

\begin{proof}
Let $\ep>0.$  It follows from Corollary 1.20 of \cite{Linalm97} that there exists $\dt_1>0$ such that
when  $\|s_is_j-s_js_i\|<\dt_1,$ there are $a_1, a_2,...,a_n \in A_{s.a.}$ 
such that
\beq
a_ia_j=a_ja_i\andeqn \|s_i-a_i\|<\ep/2\rforal 1\le i,\,j\le n.
\eneq
Let $C$ be the commutative \SCA\, generated by $a_1,a_2,...,a_n.$
Then there exists a unital \hm\,
$\phi_0: C(\I^n)\to C$ such that
$\phi_0(e_j)=a_j,$ $j=1,2,...,n.$
Hence $C\cong C(X_0)$ for some compact subset of $\I^n.$ 
Denote $\phi_1: C(X_0)\to C$  the isomorphism given above. We may 
assume that $\phi_1(e_j|{X_0})=a_j,$ $j=1,2,...,n.$ 

Put $\dt=\min\{\ep/2, \dt_1/2\}.$
Choose any $k\in \N$ with $2\sqrt{n}/k<\dt$ and $1/k\le \eta<\dt.$
By  Corollary \ref{Cc1},
 we obtain  a $1/k$-brick combination  $X\subset \I^n$ 
and a unital injective
\hm\, $\phi: C(X)\to B$   such that
\beq
\|\phi(e_j|_X)-a_j\|<\ep/2, \,\,\, j=1,2,...,n.
\eneq
Hence 
\beq
\|\phi(e_j|_X)-s_j\|<\ep,\,\,\, j=1,2,..,n.
\eneq
\end{proof}

\section{Multiple self-adjoint operators}
The main purpose of this section is to prove Theorem \ref{TTmul}.


The following lemma is a known (see the proof of Theorem 3.15 of \cite{BP}). 

\begin{lem}\label{Lprojlift}
Let $A$ be a unital \CA\, and $J$ be  an ideal of $A$ which is a $\sigma$-unital 
\CA\, of real rank zero.   Suppose that $q\in A/J$ is a projection 
such that $[q]\in \pi_{*0}(K_0(A)),$ where $\pi: A\to A/J$ is the quotient map.
Then there is a projection $p\in A$ such that $\pi(p)=q.$
%
\end{lem}

\begin{lem}\label{Ldig}
Fix $n\in \N.$ For any $\ep>0,$ there exists $\dt(n,\ep)>0$ satisfying the following:
Suppose that  $A$ is a unital \CA\, with an ideal $J$ such that:

1. $J$ is separable and has real rank zero;

2. the quotient $A/J$ is purely infinite ans simple.

3. 
$T_1, T_2,...,T_n\in A_{s.a.}$ with $\|T_i\|\le 1$ ($1\le i\le n$);

4. the $n$-tuple has 
 a $\dt$-near-spectrum  $X$ and  $(\pi(T_1), \pi(T_2),...,\pi(T_n))$ has a
$\dt$-spectrum $Y.$
Then,  there are mutually orthogonal non-zero projections 
$p_1, p_2,...,p_m\in A\setminus J,$ 
a $\dt$-dense  subset  $\{\lambda_1, \lambda_2,...,\lambda_m\}$ in $Y,$   and 
a c.p.c. map $L: C(X)\to (1-p)A(1-p),$ where $p=\sum_{k=1}^m p_k,$
such that
\beq
&&\|\sum_{k=1}^m e_j(\lambda_k)p_k+L(e_j|_X)-T_j\|<\ep,\\\label{dig-01}
&&\|L(e_ie_j|_X)-L(e_i|_X)L(e_j|_X)\|<\ep,\,\,\, 1\le i, j\le n,
\eneq
and $1-p\not\in J.$
\end{lem}

\begin{proof}
Let $\ep>0.$
Choose $\dt=\ep/128.$
Since $X$ is a $\dt$-near-spectrum of $(T_1,T_2,...,T_n),$
by Definition \ref{Dappsp},
there is a unital c.p.c. map 
$\Phi: C(X)\to A$
such that
\beq\label{dig-5}
\|\Phi(e_j|_X)-T_j\|<\dt \andeqn
\|\Phi(e_ie_j|_X)-\Phi(e_i|_X)\Phi(e_j|_X)\|<\dt,\,\,\, i,j\in \{1,2,...,n\}.
\eneq
Since $Y$ is a $\dt$-spectrum of $(\pi(T_1),\pi(T_2),..., \pi(T_n)),$
there exists, by Definition \ref{Dappsp} and by Corollary \ref{C1C},  a compact subset $Z\subset \I^n$ such that
$Y\subset Z\subset Y_\dt$ and  a unital monomorphism $\phi: C(Z)\to A/J$
such that
\beq\label{dig-6}
\|\phi(e_j|_Y)-\pi_c(T_j)\|<\dt,\,\,\, j=1,2,...,n.
\eneq

Let $e'\in A/J$ be a non-zero projection such that $1-e'\not=0.$
Choose a partial isometry 
$w\in A/J$ such that
\beq
w^*w=1\andeqn ww^*=e\le e'.
\eneq
It follows that $[1-e]=0$ in $K_0(A/J).$
Let $\{\lambda_1, \lambda_2,...,\lambda_m\}$ be 
 a $\dt$-dense subset of $Y$ which is also $2\dt$-dense in $Z.$ 
 Let $q_1, q_2,...,q_m$ be mutually orthogonal non-zero projections in $(1-e)(A/J)(1-e)$
 such that $[q_i]=0$ in $K_0(A/J).$
 Define $\psi: C(Z)\to A/J$ 
 by
 \beq
 \psi(f)=\sum_{k=1}^m f(\lambda_k)q_k+w\phi(f)w^*\rforal f\in C(Z).
 \eneq
 Note that $[\phi]=[\psi]$ in $KK(C(Z), A/J).$
Then, since $A/J$ is purely infinite and simple, by Theorem 1.7 of \cite{Dd}, there exists a unitary 
$u\in A/J$ such that
\beq\label{dig-10}
\|u\psi(e_j|_Z)u^*-\phi(e_j|_Z)\|<\dt/16\rforal j=1,2,...,n.
\eneq
Let $q_k'=uq_ku^*,$ $k=1,2,...,n,$ and 
$$
\phi_0: C(Z)\to ueu^*(A/J)ueu^*
\,\,\,{\rm by}\,\,\,
\phi_0(f)=uw\phi(f)w^*u^*\rforal f\in C(Z).
$$
Then  we may write 
\beq
u\psi(f)u^*=\sum_{k=1}^m f(\lambda_k)q_k'+\phi_0(f)\rforal f\in C(Y).
\eneq
Since $[q_k']=0$ in $K_0(A/J),$  by applying Lemma \ref{Lprojlift}, we obtain projections 
$P_k\in A$ such that 
\beq
&&\pi_c(P_k)=q_k',\,\,\, k=1,2,...,m, \,\,\, \pi_c(1-\sum_{k=1}^m P_k)=ueu^*\andeqn\\
&& P_kP_{k'}=P_{k'}P_k=0,\,\,\,{\rm if}\,\,\, k\not=k'.
\eneq
Let $P=\sum_{k=1}^m P_k.$ Then $P\in A\setminus J$ is a projection.
By the Choi-Effros lifting theorem (\cite{CE}), there is a c.p.c. map 
$\Psi: C(Y)\to (1-P)A(1-P)$ such that $\pi_c \circ \Psi=\phi_0.$

Then, by \eqref{dig-5}, \eqref{dig-6}, and \eqref{dig-10}, there are $h_j\in J,$ $j=1,2,...,n,$ such that
\beq\label{Ldig-20}
\|\Phi(e_j|_X)-(\sum_{k=1}^m e_j(\lambda_k)P_k+\Psi(e_j|_Z))-h_j\|<2\dt+\dt/16,\,\,\, j=1,2,...,n.
\eneq

Let $\{d_{k, l}\}$ be an approximate identity for  $P_kJP_k,$ $k=1,2,...,m,$
and $\{d_{0, l}\}$ an approximate identity for $(1-P)J(1-P)$ consisting of projections.
Put 
$$
d_l=d_{0,l}+\sum_{k=1}^m d_{k,l}\andeqn
d_l'=\sum_{k=1}^m d_{k,l},\,\,\,l\in \N.
$$ Then 
$\{d_l\}$ is an approximate identity for $J$ 
and $\{d_l'\}$ is an approximate identity for $PJP.$
Choose $l$ such that
\beq
\|(1-d_l)h_j\|<\dt/64\andeqn \|h_j(1-d_l)\|<\dt/64,\,\,\, j=1,2,...,n.
\eneq
Hence, since $P-d_l'\le 1-d_l,$ 
\beq
\|(P-d_l')h_j\|<\dt/64\andeqn \|h_j(P-d_l')\|<\dt/64,\,\,\, j=1,2,...,n.
\eneq
Put $p_k=P_k-d_{k,l},$ $k=1,2,...,m.$
Then, by \eqref{Ldig-20},  for $j=1,2,...,n,$
\beq
(P-d_l')\Phi(e_j|_X)&&\approx_{33\dt/16} (P-d_l')(\sum_{k=1}^m e_j(\lambda_k)P_k+\Psi(e_j|_Y)+h_j)\\
&&\approx_{\ep/64}\sum_{k=1}^m e_j(\lambda_k)p_k.
\eneq
Similarly, for $j=1,2,...,n,$
\beq
\Phi(e_j|_X)(P-d_l')&&\approx_{33\dt/16} (\sum_{k=1}^m e_j(\lambda_k)P_k+\Psi(e_j|_Y)+h_j)(P-d_l')\\
&&\approx_{\dt/64}\sum_{k=1}^m e_j(\lambda_k)p_k.
\eneq
Put $p=P-d_l'.$ 
Then
\beq
&&\|p\Phi(e_j|_X)-\sum_{k=1}^m e_j(\lambda_k)p_k\|<33\dt/16+\dt/64,\\
&&\|p\Phi(e_j|_X)p-\sum_{k=1}^m e_j(\lambda_k)p_k\|<33\dt/16+\dt/64\andeqn\\\label{NNN-1}
&&\|p\Phi(e_j|_X)-\Phi(e_j|_X)p\|<33\dt/8+\dt/32,\,\,\,\,\,\, j=1,2,...,n.
\eneq
Put $L(f)=(1-p)\Phi(f)(1-p)$ for $f\in C(X).$ Then it is a c.p.c. map. Moreover
\beq
\Phi(e_j|_X)&=&p\Phi(e_j|_X)+(1-p)\Phi(e_j|_X)\\\label{NNN}
&&\approx_{33\dt/8+3\dt/64} \sum_{k=1}^m e_j(\lambda_k)p_k
+(1-p)\Phi(e_j|X)(1-p).
\eneq
Hence, using the inequality above and \eqref{dig-5}, 
\beq\nonumber
&&\hspace{-0.5in}\|\sum_{k=1}^m e_j(\lambda_k)p_k+L(e_j|_X)-T_j\|
\le \|\sum_{k=1}^m e_j(\lambda_k)p_k+L(e_j|_X)-\Phi(e_j|_X)\|+
\|\Phi(e_j|_X)-T_j\|\\
&&\hspace{1.5in}<33\dt/8+3\dt/64+\dt<\ep,\,\,\,\,\,\, j=1,2,...,n.
\eneq
To arrange  $(1-p)\not\in J,$ we may split $p_1.$ We write $p_1=p_1'+p_1'',$
where $p_1', p_1''\in p_1Jp_1$ are two mutually orthogonal projections and both 
are not in $J$ as $p\not\in {\cal K}.$
Then define $L': C(X)\to (1-p+p_1'')A(1-p+p_1'')$ by
$L'(f)=f(\lambda_1)p_1''+L(f)$  for all $f\in C(X).$ 
We then replace $L$ with $L'$ and 
replace $p_1$ with  $p_1'.$

Finally,  by \eqref{NNN-1} and \eqref{dig-5}, 
\beq\nonumber
&&\hspace{-0.3in}L(e_ie_j|_X)-L(e_i|_X)L(e_j|_X)=(1-p)\Phi(e_ie_j|_X)(1-p)-(1-p)\Phi(e_i|_X)(1-p)\Phi(e_j|_X)(1-p)\\\nonumber
&&\hspace{0.3in}\approx_{33\dt/8+\dt/32}(1-p)(\Phi(e_i|_X)-\Phi(e_i|_X)\Phi(e_j|_X))(1-p)\approx_\dt 0.
\eneq
Thus \eqref{dig-01} also holds.
\end{proof}

The following proposition is a list of easy facts. We refer to Definition 2.1 of \cite{GL}
for the definition of these properties.  It is certainly known that $M(A\otimes {\cal K})$ 
has these properties and some of these may also stated in various places. We put here for our convenience.

\begin{prop}\label{Pkkr}
Let $A$ be a $\sigma$-unital \CA. Then $M(A\otimes {\cal K})$ has the following property

(1) $K_0$-$r$-cancellation with $r(n)=1$ for all $n\in \N;$ 

(2)  $K_1$-$r$-cancellation with $r(n)=1$ for all $n\in \N;$

(3) $K_1$-stable rank 1;

(4) $K_0$-stable rank 1;

(5) stable exponential length $b=3\pi;$

(6) stable exponential rank $4;$

(7) $K_0$-divisible rank $T(n,m)=1$ (for $(n,m)\in \N^2$);

(8) $K_1$-divisible rank $T(n,m)=1$ (for $(n,m)\in \N^2$), and 

(9) exponential length divisible rank $E(r, n)=3\pi$ (for all $r\in \R_+, n\in \N.$).

\end{prop}

\begin{proof}
First, by \cite{CH}, $K_i(M(A\otimes {\cal K}))=\{0\},$ $i=0,1.$
We also have, for any $k\in \N,$ $M_k(M(A\otimes {\cal K}))\cong M(A\otimes {\cal K}).$ 
For any projection $p\in  M(A\otimes {\cal K}),$  we have $p\oplus 1_{M(A\otimes {\cal K})}\sim 
1_{M(A\otimes {\cal K})}.$
This immediately implies (1), (4) and (7). 

It follows from Theorem 1.1 of \cite{ZhJOT92} that ${\rm cel}(M(A\otimes {\cal K}))\le 3\pi$
and ${\rm cer}(M(A\otimes {\cal K}))\le 4.$
Therefore, (2), (3), (5), (6), (8) and (9) hold. 
\end{proof}

\begin{df}\label{DR}
Let $A$ be a unital \CA. 
Let $r, b: \N\to \N$ and $T: \N^2\to \N$ be maps, and  $R,\,s>0.$
We say $A$ satisfies condition $(R, r, b, T, s)$ if 
$A$ has $K_i$-$r$-cancellation,
$i=0,1,$ 
$K_1$-stable rank $s,$ $K_0$-divisible rank $T,$  ${\rm cer}(M_m(A))\le R$ for all $m\in \N,$ 
and ${\rm cel}(M_m(A))\le b(m)$ for all $m\in \N.$
\end{df}

If $A$ is a  unital purely infinite simple \CA, then $A$ satisfies condition $(R, r, b,T, s)$
for $R=2,$ $s=1,$ any $r$ and $T$ and $b=2\pi.$  If $A$ has stable rank one and real rank zero, 
then $A$ satisfies condition $(R, r, b, T, s)$ for $R=2,$ $s=1,$ any $r,$  $T, $ and $b=2\pi.$


We will apply the following theorem which is a variation of 
Theorem 3.1 of \cite{GL}.

\begin{thm}\label{Tgl99}
Fix $n\in \N.$
Let $r, b: \N\to \N, T : \N^2\to \N$  be maps,  and $R>0,\,
s > 0.$ 

For any $\ep > 0$  
there exist a positive number 
$\dt> 0$ 
and an integer $l > 0$ satisfying the following: 
Suppose that $A$ is a 
unital \CA\, such that:

1. $A$ satisfies   condition $(R, r, b,T, s)$;

2. $K_i(A\otimes C(Y))=0,$ $i=0,1,$ for any compact metric space $Y.$ 

Let $X$ be a compact subset of $\I^n,$ 
and let $L_1, L_2: C(X)\to A$  be
 unital  c.p.c. maps such that
 \beq\label{Lgl99-1}
 \|L_k(e_ie_j|_X)-L_k(e_i|_X)L_k(e_j|_X)\|<\dt,\,\, \, 1\le i,j\le n, \,\,\, k=1,2.
 \eneq
 %
Then, there are a unitary $u\in  M_{l+1}(A)$ (for some $l\in \N$) and a homomorphism $\sigma : C(X) \to  M_l(A)$ with
finite dimensional image, such that
\beq
\|u^*{\rm diag}(L_1(e_j|_X), \sigma(e_j|_X))u - {\rm diag}(L_2(e_j|_X), \sigma(e_j|_X))\|<\ep,\,\,\,
1\le j\le n.
\eneq
\end{thm}

\begin{proof}
Let us first fix a compact subset $X.$ 
Let $\ep>0$ and ${\cal F}=\{e_j|_X: 0\le j\le n\}.$
Since we assume that 
$K_i(A\otimes C(Y))=0$ for any compact metric space $Y,$ $i=0,1,$ 
for any finite subset ${\cal P}$ as in Theorem 3.1 of \cite{GL}, we have that 
$[L]|_{\cal P}=0$ for any sufficiently multiplicative c.p.c. maps
$L: C\to A.$ 
Then, by Proposition \ref{Pkkr} and applying  Theorem 3.1 of \cite{GL},
we obtain  $\dt_0$ and a finite subset ${\cal G}_0\subset C(X)$ 
such that, if 
\beq
\|L_i(fg)-L_i(f)L_i(g)\|<\dt_0\rforal f,g\in {\cal G}_0,\,\,\, i=1,2,
\eneq
there exists an integer $l>0,$  a unital \hm\, $\sigma: C(X)\to  M_l(A)$ with finite dimensional range and 
a unitary $u\in M_l(A)$ such that
\beq
\|u^*{\rm diag}(L_1(e_j|_X), \sigma(e_j|_X))u - {\rm diag}(L_2(e_j|_X), \sigma(e_j|_X))\|<\ep/2,\,\,\,
1\le j\le n
\eneq
(Note  as mentioned above,  that  $[L_i]|_{\cal P}=0,$ whenever it makes sense, since $K_i(A\otimes C(Y))=0$
for any compact metric space $Y,$ $i=0,
1.$)

However, since $\{e_j: 0\le j\le n\}$ generates $C(X)$ as \CA,
one obtains $\dt(n, \ep)>0$ such that  when   \eqref{Lgl99-1} holds,  then
\beq
\|L_i(fg)-L_i(f)L_i(g)\|<\dt_0\rforal f,g\in {\cal G}_0,\,\,\, i=1,2,
\eneq
So Theorem 3.1 of \cite{GL} applies and 
the theorem would follow if we fix $X$ first.   

To find a common $\dt$ and $l,$ indepedent of $X\subset \I^n,$ we note that there are finitely many compact subsets 
$X_1, X_2,...,X_m$ of $\I^n$  such that, for any non-empty compact subset $X,$
there is $s\in \{1,2,...,m\}$ such that 
\beq
d_H(X, X_s)<\ep/6
\eneq
(see \ref{DHd}).
Let $\Omega_s=\{x\in \I^n: {\rm dist}(x, X_s)\le  \ep/6\},$ $s=1,2,...,m.$  Then, for any non-empty 
compact subset $X$ of $\I^n,$  there is $s\in \{1,2,...,m\}$ 
such that
\beq
d_H(X, \Omega_s)\le \ep/3 \andeqn X\subset \Omega_s.
\eneq
When $\ep>0$ is given,
by the first part of the proof, we obtain $\dt_s>0$ and $l_s\in \N$ 
such that the conclusion of the first part holds for   $\Omega_s,$ $s=1,2,..., m.$
Choose $\dt=\min\{\dt_1, \dt_2,...,\dt_m\}$ and $l=\max \{l_1,l_2,...,l_m\}.$

Now suppose $X\subset \I^n$ and $L_1, L_2: C(X)\to A$ 
are unital c.p.c. maps satisfying  the assumption \eqref{Lgl99-1}.
Suppose that 
\beq
d_H(X,\Omega_s)\le \ep/3\andeqn X\subset \Omega_s.
\eneq
Define $\tilde L_k: C(\Omega_s)\to A$ by 
$\tilde L_k(e_j|_{\Omega_s})=L_k(e_j|_X)$ (this makes sense as $X\subset \Omega_s$),
$j=1,2,...,n,$ $k=1,2.$  Therefore
\beq
\|\tilde L_k(e_ie_j|_{\Omega_s})-\tilde L_k(e_i|_{\Omega_s})\tilde L_k(e_j|_{\Omega_s})\|<\dt,\,\,i,j\in \{1,2,...,n\}
\eneq
and $k=1,2.$  Hence, with the choice of $\dt,$ 
we obtain \hm\, $\sigma': C(\Omega_s)\to M_l(A)$ with finite dimensional range
and a unitary $u\in M_{l+1}(A)$  such that, for $j=1,2,...,n,$ 
\beq\label{Tgl99-10}
\|u^*{\rm diag}(L_1(e_j|_{\Omega_s}), \sigma'(e_j|_{\Omega_s}))u -
 {\rm diag}(L_2(e_j|_{\Omega_s}), \sigma'(e_j|_{\Omega_s}))\|<\ep/2
\eneq
We may write $\sigma'(e_j|_{\Omega_s})=\sum_{k=1}^K e_j(\xi_k) p_k,$
where $\xi_k\in \Omega_s$ ($k=1,2,...,K$) and 
$\{p_1,p_2,...,p_K\}$ is a set of mutually orthogonal projections in $M_l(A).$
Since  $d_H(X, \Omega_s)\le \ep/4,$
for each $\xi_k,$ we may choose $x(\xi_k)\in X$ such that
\beq
\|\xi_k-x(\xi_k)\|_2={\rm dist}(\xi_k, x(\xi_k))\le \ep/3,\,\,1\le k\le K
\eneq
Define  $\sigma: C(X)\to M_{l+1}(A)$ by 
\beq
\sigma(f)=\sum_{k=1}^K f(x(\xi_k))p_k \rforal f\in C(X).
\eneq
Note that 
\beq
\|\sigma(e_j|_X)-\sigma'(e_j|_{\Omega_s})\|\le \ep/3,\,\,\, 1\le j\le n.
\eneq
Hence, by \eqref{Tgl99-10}, for $1\le j\le n,$
\beq
\|u^*{\rm diag}(L_1(e_j|_{X}), \sigma(e_j|_{X}))u -
 {\rm diag}(L_2(e_j|_{X}), \sigma(e_j|_{X}))\|\le \ep/2 +\ep/3<\ep.
\eneq
%
%
\end{proof}

\begin{thm}\label{TTmul}
Let $n\in \N$ and  $\ep>0.$   Let  $R\in \R_+\setminus \{0\},$ $r, b: \N\to \N$ and $T: \N^2\to \N$ be maps and $s>0.$ 
There exists $\dt(n, \ep)>0$ satisfying the following:

Suppose that $A$ is a unital \CA\,  such that:

1.   $A$ satisfies condition $(R, r, b, T, s);$

2. $K_i(A\otimes C(Z))=0$ for any compact metric space $Z$ ($i=0,1$);

3. $J\subset A$ is an essential ideal which has real rank zero  such that $A/J$ is purely infinite simple;

4. for any $l\in \N$  and any non-zero projection $e\in M_l(A)\setminus M_l(J),$  we have 
$e\sim 1.$ 

Suppose further that 
$T_1, T_2,...,T_n\in A_{s.a.}$ 
with $\|T_i\|\le 1$  
($1\le i\le n$) 
satisfy :
\beq\label{TTmul-0}
\|T_iT_j-T_jT_i\|<\dt,\,\,\, 1\le i,j\le n\tand
d_H(X, Y)<\ep/8
\eneq
(recall Definition \ref{DHd}),
where $X=s{\rm Sp}^{\ep/8}(T_1, T_2,...,T_n),$ 
$Y=s{\rm Sp}^{\ep/8}((\pi(T_1), \pi(T_2),...,\pi(T_n)),$  and 
where $\pi: A\to A/J$ is the quotient map.

Then there are $S_1, S_2,...,S_n\in A_{s.a.}$
such that
\beq
S_iS_j=S_jS_i,\,\,\, 1\le i,j\le n\andeqn \|T_j-S_j\|<\ep,\,\,\, 1\le j\le n.
\eneq
\end{thm}

\begin{proof}
Fix $n\in \N.$
Let $\ep>0.$ 
%
%
%
%
We will apply Theorem \ref{Tgl99}. 
 Let  
 $\dt_1>0$  (in place of $\dt$)
 be 
     given  by  Theorem \ref{Tgl99}  for $\ep/16.$

Put $\ep_0=\min\{\ep/128,\dt_1/4\}.$
Let  $\dt_2'>0$ (in place of $\dt$)  be  given by Lemma \ref{Ldig} for $\ep_0$
(in place of $\ep$).  
Choose $\dt_2=\min \{\dt_2', \ep_0/2\}.$
Put $\ep_1=\min\{\ep_0/2, \dt_2/4\}.$


Let $\dt_3>0$ be given by Proposition \ref{Pspuniq}  for $\ep_1$
(in place of $\eta$). 
Choose $\ep_2=\{\ep_1/2, \dt_3/2\}.$ 
Let $\dt_4>0$ be given by Proposition \ref{Pappsp} for $\ep_2$ (in place of $\eta$)
and let $\dt_5>0$ be given by Theorem \ref{s2-T1} for $\ep_2.$ 

Let $\dt:=\min\{\dt_1/2, \dt_2/2, \dt_3/2, \dt_4, \ep_2\}.$ 

Now suppose that $T_1, T_2,...,T_n\in A_{s.a.}$ with $\|T_i\|\le 1$ ($1\le i\le n$) such that
\beq
\|T_iT_j-T_jT_i\|<\dt\rforal 1\le j\le n\andeqn d_H(X, Y)<\ep/8,
\eneq

By the choice of $\dt$ (and $\dt_4,$ $\dt_5$) and, applying Proposition \ref{Pappsp}, we obtain 
non-empty sets 
 $X=s{\rm Sp}^{\ep/8}((T_1, T_2, ...,T_n))$ 
and $Y=s{\rm Sp}^{\ep/8}((\pi(T_1), \pi(T_2),...,\pi(T_n))$  (see also \eqref{pSp<}),
and the $n$-tuple $(T_1, T_2,..., T_n)$
has  an
$\ep_2$-near-spectrum $X_1$ and, by Theorem \ref{s2-T1}, 
$(\pi(T_1), \pi(T_2),...,\pi(T_n))$ has an  $\ep_2$-spectrum $Y_1.$
Moreover (by the choice of $\dt_3$), since $\ep_2=\min\{\ep_1/2, \dt_3/2\},$
we have  (by applying Proposition \ref{Pspuniq}), 
\beq\label{TTmul-9}
Y_1\subset \overline{(X_1)_{\ep_1}},\,\,\,  X_1\subset X\subset \overline{(X_1)_{2\ep_1}},
 \andeqn Y_1\subset Y\subset \overline{(Y_1)_{2\ep_1}}.
\eneq

By the choice of $\dt_2$ (and $\ep_2<\dt_2'$),  applying Lemma \ref{Ldig}, we obtain mutually orthogonal nonzero projections 
$p_1,p_2,...,p_m\in A\setminus J,$  a $\dt_2$-dense subset 
$\{\lambda_1, \lambda_2,...,\lambda_m\}\subset Y_1,$ 
a c.p.c. map  $L: C(X_1)\to (1-p)A(1-p)$ such that (with $p=\sum_{i=1}^m p_i$)
\beq\label{TTmul-10}
&&\|\sum_{i=1}^m e_j(\lambda_i)p_i+L(e_j|_{X_1})-T_j\|<\ep_0<\dt_1\andeqn\\\label{TTmul-15}
&&\|L(e_ie_j|_X)-L(e_i|_X)L(e_j|_X)\|<\ep_0,
\eneq
$j=1,2,...,n,$
and $1-p\not\in J.$  
Put 
$$
\Phi_1(f)=\sum_{i=1}^m f(\lambda_i)p_i+L(f|_{X_1})\rforal f\in C(\I^n).
$$
%
Note that, since $(1-p)\sim 1,$  $(1-p)A(1-p)\cong A.$ 
Choose $\xi_0\in X_1$ and 
define $\Phi_0: C(X_1)\to (1-p)A(1-p)$ by $\Phi_0(f)=f(\xi_0)(1-p).$

By the choice of $\dt_1$ and \eqref{TTmul-15}, as $(1-p)A(1-p)\cong A,$
and applying Theorem \ref{Tgl99},
we obtain a unital \hm\, $\sigma: C(X_1)\to M_l((1-p)A(1-p))$ 
(for some integer $l\in \N$) with finite dimensional range and 
a unitary $u_1\in M_{l+1}((1-p)A(1-p))$
such that
\beq\label{TTmul-16}
\|u_1^*\diag(\Phi_0(e_j|_{X_1}), \sigma(e_j|_{X_1}))u_1-\diag(L(e_j|_{X_1}), \sigma(e_j|_{X_1}))\|<\ep/16,
\,\,\, 1\le j\le n.
\eneq
Since $\sigma$ has a finite dimensional range, there are mutually orthogonal projections 
$q_1,q_2,...,q_L\in M_{l}((1-p)A(1-p))$ (for some integer $L\ge 1$) such that 
\beq
\sigma(e_j|_{X_1})= \sum_{k=1}^L e_j(\xi_k)q_k,\,\,\, 1\le j\le n,
\eneq
where $\xi_k\in X_1,$ $k=1,2,...,L.$ Put $q=\sum_{i=1}^L q_i.$ 
Note that, without loss of generality, by 
replacing $l$ by $l+1$ and adding a projection $q_i''\in M_l((1-p)A(1-p))\setminus M_l((1-p)J(1-p))$
to each $q_i,$  if neseccarily, we may assume that $q_i\not\in M_l((1-p)A(1-p))\setminus M_l((1-p)J(1-p)),$
$i=1,2,...,L.$

Define $\Phi_2: C(\I^n)\to M_{l+1}((1-p)A(1-p))$ by 
\beq\label{TTmul-16+}
\Phi_2(f)=u_1^*\diag(\Phi_0(f|_{X_1}), \sigma(f|_{X_1}))u_1\rforal f\in C(\I^n).
\eneq

Let $\{e_{i,j}\}_{1\le i,j\le l+1}$ be a chosen matrix unit for $M_{l+1}(A).$
In what follows, we identify $1-p$ with $(1-p)\otimes e_{1,1}$ and 
$M_l(A)$ above with $EM_{l+1}(A)E,$ where $E=\sum_{i=2}^{l+1}1\otimes e_{i,i}.$
In particular, we view $q\in  EM_{l+1}(A)E.$

Recall that $\{\lambda_1, \lambda_2,...,\lambda_m\}$ is $\dt_2$-dense in $Y_1.$
By the assumption that $d_H(X,Y)<\ep/8$ and \eqref{TTmul-9},  we conclude that
$\{\lambda_1, \lambda_2,...,\lambda_m\}$  is $\dt_2+2\ep_1+\ep/8$-dense in $X_1.$
Note  that 
\beq
\dt_2+2\ep_1+\ep/8<\ep/256+\ep/128+\ep/8<35\ep/256.
\eneq

Hence  $\{\lambda_1, \lambda_2,...,\lambda_m\}$ is $35\ep/256$-dense
in $X_1.$
Let 
\beq
F_1&=&
\{\xi_1,\xi_2,...,\xi_L\}\cap B(\lambda_1,35\ep/256),\\
%
F_2&=&(\{\xi_1, \xi_2,...,\xi_L\}\setminus F_1)\cap B(\lambda_2, 35\ep/256).
\eneq
By induction, since $\{\lambda_1, \lambda_2,...,\lambda_m\}$ is $35\ep/256$-dense
in $X_1,$
 we obtain mutually disjoint subsets 
$F_1, F_2,..., F_{m'}$ such that $\sqcup_{i=1}^{m'} F_i=\{\xi_1, \xi_2,...,\xi_L\},$
and $F_j\subset B(\lambda_j, 35\ep/256),$ $j=1,2,...,m',$
where $m'\le m.$

Put $q_k'=\sum_{\xi_j\in S_k}q_j,$ $k=1,2,...,m'.$
Define 
\beq
\sigma_1(e_j|_X)=\sum_{k=1}^{m'} e_j(\lambda_k)q_k'.
\eneq
Then 
\beq\label{TTmul-17+}
\|\sigma_1(e_j|_X)-\sigma(e_j|_X)\|<35\ep/256,\,\,\, 1\le j\le n.
\eneq
Recall that, $e\sim 1$ for any nonzero projection $e\in M_{l+1}(A)\setminus M_{l+1}(J).$ 
Thus, in  $M_{l+1}(A),$ viewing $q\in EM_{l+1}(A)E,$ we have a partial isometry 
$w$ such that $qwp=w,$ 
\beq\nonumber
w^*q_k'w=p_k,\,\,\, k=1,2,...,m'\andeqn w_1^*\sigma_1(e_j|_X)w_1=\sum_{k=1}^{m'} e_j(\lambda_k)p_k,\,\,\, 1\le j\le n.
\eneq
Then put  $w_1=(1-p)+ w.$  
Then 
\beq
w_1^*w_1=(1-p) +\sum_{k=1}^{m'}p_k\,\,\,(\le (1-p)+p\le 1\otimes e_{1,1}).
\eneq
It follows that
\beq\label{TTmul-30}
\Phi_1(e_j|_X)&=& w_1^*\diag(L(e_j|_X), \sigma_1(e_j|_X))w_1\oplus \sum_{k=m'}^m e_j(\lambda_k)p_k\\\label{TTmul-30+}
&=& L(e_j|_X)+\sum_{k=1}^m e_j(\lambda_k)p_k,
\,\,\, 1\le j\le n.
\eneq
Put 
\beq
\psi_0(f)=\sum_{k=m'}^m f(\lambda_k)p_k\andeqn
\Phi_3(f)=w_1^*\diag(L(f|_{X_1}), \sigma(f|_{X_1}))w_1
\eneq
for all $f\in C(\I^n).$
Then, by \eqref{TTmul-16}, \eqref{TTmul-16+}, \eqref{TTmul-17+}, \eqref{TTmul-30+} and \eqref{TTmul-10},  for $1\le j\le n,$
\beq\label{TTmul-20}
w_1^*\Phi_2(e_j)w_1\oplus \psi_0(e_j)\approx_{\ep/16} 
\Phi_3(e_j)
\oplus \psi_0(e_j)
\approx_
{35\ep/256}\Phi_1(e_j|_{X_1})\approx_{\ep_0} T_j.
\eneq
Put 
\beq
S_j=w_1^*\Phi_2(e_j)w_1\oplus \psi_0(e_j),\,\,\, 1\le j\le n.
\eneq
Then, since $\ep/16+35\ep/256+\ep_0<\ep,$  by \eqref{TTmul-20},
\beq
\|S_j-T_j\|<\ep,\,\,\, 1\le j\le n.
\eneq
Since $\Phi_0$ is a \hm\, so is $\Phi_2.$ Hence
\beq
S_iS_j=S_jS_i,\,\,\, 1\le i,j\le n.
\eneq
\end{proof}

\vspace{0.1in}

{\bf Proof of Theorem \ref{TTTmul}}

\begin{proof}
Let $A=B(H)$ and $J={\cal K}.$ By \cite{CH}, $K_i(B(H)\otimes C(Y))=0$
for all compact metric space $Y,$ $i=0,1.$ By Proposition \ref{Pkkr}, we may apply Theorem \ref{TTmul}.
\end{proof}

\vspace{0.1in}

\begin{NN}{\bf  Proof of Theorem \ref{TTTmodule}}\label{47}\end{NN}
\begin{proof}
Let $B$ be a $\sigma$-unital purely infinite simple \CA\,
and $H$ be a countably generated Hilbert $B$-module.
Let $H_B=\{\{b_n\}: \sum_{n=1}^\infty b_n^*b_n \,\,\, {\rm converges\,\, in\,\, norm}\}.$
Then, by  Kasparov's absorbing 
theorem (Theorem 2 of  \cite{Kas}), $H$ is an orthogonal summand of 
$H_B$ and hence $K(H)$ is isomorphic to a hereditary \SCA\, of $B\otimes {\cal K}$
which is purely infinite simple.
If $K(H)$ is unital, by Theorem 1 of \cite{Kas}, $L(H)=K(H).$ Then the conclusion of 
Theorem \ref{TTTmodule} follows from Theorem \ref{s2-T1} (without referring to synthetic 
spectra). 

Otherwise, since $H$ is countably generated, 
$C:=K(H)$ is non-unital but $\sigma$-unital (see Proposition 3.2 of \cite{BL}).  
By Theorem 1 of \cite{Kas},
$L(H)\cong M(K(H))=M(C).$  
Note that $B\cong B\otimes {\cal K}$ (see Remark 2.5 (c) of \cite{LZ})
which has real rank zero  by \cite{Zhpif}.  It follows from Theorem 3.3 of \cite{Zhsimplecor} 
(see also  Remark 2.5 (b) of \cite{LZ}) that $M(B)/B$
is purely infinite simple, and every projection in $M(B)\setminus B$ is equivalent to $1_{M(B)}.$
Thus, by Proposition \ref{Pkkr}, Theorem \ref{TTmul} applies. 
\end{proof}

\section{A pair of almost commuting self-adjoint operators on Hilbert modules}

\begin{df}[Definition 3.1 of \cite{FR}]\label{DIR}
For a unital \CA\, $A,$ denote by $R(A)$ the set of elements $x\in A$ with the property that for no
ideal $I$ of $A$ is $x+I$ one-sided and not two-sided invertible in $A/I.$
Recall that $A$ is said to have property (IR) if all elements in $R(A)$ belongs to the norm closure of 
$GL(A),$ the group of invertible elements. 

A non-unital \CA\, $A$ has property (IR), if the \CA\, obtained by adjoining a unit to $A$ has property 
(IR).  

It follows from Lemma 4 of \cite{LinWFN} that, for nonzero any projection $p\in A,$
$pAp$ has (IR).

Every \CA\, of stable rank one has property (IR), and every
purely infinite simple \CA\, has property (IR).
\end{df}

\begin{lem}\label{3L3}
Let $\ep>0.$ There exists $\dt(\ep)>0$ satisfying the following:
Let $B$ be a unital purely infinite simple \CA,  $a_1, a_2\in B$ be self-adjoint elements
such that  $\|a_i\|\le 1,$ $i=1,2,$ and 
\beq
\|a_1a_2-a_1a_2\|<\dt.
\eneq
Then there exists a pair of commuting self-adjoint elements $s_1, s_2\in B$
such that
\beq\label{3L3-01}
\|s_i-a_i\|<\ep\andeqn \kappa_1(\lambda-(s_1+is_2))=\kappa_1(\lambda-(a_1+ia_2))
\eneq
(recall Definition \ref{Dind}), whenever ${\rm dist}(\lambda, {\rm sp}(s_1+is_2))\ge \ep.$
\end{lem}
(Note that the identity in \eqref{3L3-01} also implies that $\lambda\not\in {\rm sp}(a_1+ia_2).$)

\begin{proof}
It follows from Theorem \ref{s2-T1} that  exists $\dt>0$ such that,
whenever
\beq
\|a_1a_2-a_2a_1\|<\dt,
\eneq
there exists a compact subset $X\subset \I^2$ 
and a unital injective \hm\, $\phi: C(X)\to B$
such that
\beq
\|\phi(e_i|_X)-a_i\|<\ep/2,\,\,\, i=1,2.
\eneq
Put $s_i=\phi(e_i|_X),$ $i=1,2,$ and $b=s_1+is_2.$
Then $b$ is a normal element in $B$ and ${\rm sp}(b)=X.$
Moreover
\beq\label{3L3-5}
\|b-(a_1+i a_2)\|<\ep.
\eneq
Suppose  that $\lambda\in \C$ such that
${\rm dist}(\lambda, X)\ge \ep.$
Then 
\beq
\|(\lambda-b)^{-1}\|={1\over{{\rm dist}(\lambda, X)}}\le 1/\ep.
\eneq
Hence (also by \eqref{3L3-5}) 
\beq
\|1-(\lambda-b)^{-1}(\lambda-(a_1+ia_2))\|<1.
\eneq
It follows that $(\lambda-b)^{-1}(\lambda-(a_1+ia_2))\in GL_0(B).$ 
Hence $(\lambda-b)$ and $(\lambda-(a_1+i a_2))$ are in the same path connected component 
of $GL(B).$ It follows that $\kappa_1(\lambda-b)=\kappa_1(\lambda-(a_1+ia_2)).$
%
%
%
\end{proof}

\begin{thm}\label{MT-pair+}
Let $\ep>0.$ There exists $\dt(\ep)>0$ satisfying the following:
Suppose that  $A$ is a unital \CA\,  
with an essential ideal $J$  such that:

1. $J$ is $\sigma$-unital, 
has  real rank zero and satisfies  property (IR);

2. the quotient 
$A/J$ is purely infinite simple.

\noindent
Let $T_1, T_2\in A_{s.a.}$ 
with $\|T_i\|\le 1$ ($i=1,2$) such that:
\beq\label{MT-pair-0}
\|T_1T_2-T_2T_1\|<\dt\tand \kappa_1(\lambda-(\pi(T_1+iT_2)))=0
\eneq
for all $\lambda\not\in s{\rm Sp}^\dt(\pi(T_1+iT_2)),$
where $\pi: A\to A/J$ is the quotient map.

\noindent
Then, if either:

 $\bullet$ $A$ has real rank zero, or 
 
 $\bullet$
$s{\rm Sp}^\dt (\pi(T_1+iT_2))$ is path connected,

\noindent
there are $S_1, S_2\in A_{s.a}$
such that 
\beq
S_1S_2=S_2S_1\andeqn \|S_i-T_i\|<\ep,\,\,\, i=1,2.
\eneq

\end{thm}

\begin{proof}
Fix $0<\ep<1.$
Let $\dt_1(\ep/32)>0$ be given by Theorem 4.4  of \cite{FR}
%
 for  $\ep/32$ in place of $\ep.$

Choose $\ep_1=\min\{\ep/64, \dt_1/64\}.$ 
Let $\dt_2(\ep_1)>0$ be given  by Proposition \ref
{Pspuniq} 
for $\ep_1$ (in place of $\eta$ ) and $k=2.$
Choose $\ep_2=\min\{\ep_1/32, \dt_2/32\}.$

Let $\dt_3:=\dt(\ep_2/2)>0$ be given by  Theorem \ref{3L3} for $\ep_2/2$ (in place of $\ep$).
Put $\ep_3:=\min\{\dt_3, \ep_2/4\}>0.$ 

Let  $\dt_4:=\dt(\ep_3)>0$ be given  by Proposition \ref
{Pesspsp} 
for
$\ep_3$ (in place of $\eta$).

Choose $\dt=\min\{\dt_4, \dt_3, \ep_2/4, \dt_2/2,\ep_1/4\}>0.$


Now assume that $T_1, T_2\in A_{s.a.}$ satisfy the assumption for 
$\dt.$
By Proposition \ref{Pesspsp}, $s{\rm Sp}^{\ep_3}(\pi(T_1), \pi(T))\not=\emptyset.$
Put  (see \eqref{pSp<})
$$
Z:=s{\rm Sp}^{\ep_1}((\pi(T_1), \pi(T_2))\supset s{\rm Sp}^{\ep_3}(\pi(T_1), \pi(T))
$$
Recall that $A/J$ is a purely infinite simple \CA.
Therefore, by the choice of $\dt$ and applying  Theorem \ref{3L3}
we obtain  
a 
compact subset $X\subset \I^2,$ 
a unital injective \hm\, $\phi: C(X)\to   A/J$
such that
\beq\label{MT-pair-7}
\|{\bar s_j}-\pi(T_j)\|<\ep_2/2,\,\,\, j=1,2, \andeqn
\kappa_1(\lambda-t)=\kappa_1(\lambda-\pi(T_1+iT_2))
\eneq
for those $\lambda\in \I^2$ for which ${\rm dist}(\lambda, X)\ge \ep_2/2,$
where ${\bar s}_j=\phi(e_j|_X),$ $j=1,2,$ and $t={\bar s}_1+i {\bar s}_2.$
In particular, the pair $(\pi(T_1), \pi(T_2))$ has an  $\ep_2/2$-spectrum 
$X.$ We may view it as an 
$nSp^{\dt_2}(\pi(T_1), \pi(T_2)).$
On the other hand, by the choice of $\dt_2$ and applying Proposition \ref{Pspuniq},  we have
\beq\label{MT-pair-7+1}
X\subset  Z\subset X_{2\ep_1}.
\eneq

We claim: Suppose that $\Omega\subset \C\setminus X$ is a bounded path connected component
and that there is $\xi\in \Omega$ such that ${\rm dist}(\xi, X)\ge 2\ep_1.$
Then, for any $\lambda\in \Omega,$  
\beq
\kappa_1(\lambda-({\bar s_1}+i{\bar s}_2))=0.
\eneq
To see the claim, we note that,  by \eqref{MT-pair-7+1},
$\xi\not\in Z.$   Then, by \eqref{MT-pair-7} and by the assumption \eqref{MT-pair-0}, 
\beq
\kappa_1(\xi-({\bar s_1}+i{\bar s}_2))=\kappa_1(\xi-(\pi_c(T_1)+i\pi_c(T_2)))=0.
\eneq
It follows that, for any $\lambda\in \Omega,$
\beq
\kappa_1(\lambda-({\bar s_1}+i{\bar s}_2))=\kappa_1(\xi-({\bar s_1}+i{\bar s}_2))=0
\eneq
and the claim follows.
%

Choose $k\in \N$ such that
$2\sqrt{2}/k<\ep_2/16.$ 

Note that, if $Z$ is path connected,  then $X_{2\ep_1}=\cup_{x\in X}B(x, 2\ep_1)$ is also path connected.\footnote{
If $y_i\in B(x_i, 2\ep_1),$  for some $x_i\in X\subset Z,$ $i=1,2,$
then there is a path in $Z\subset X_{2\ep_1}$ connecting $x_1$ to $x_2,$ 
and there is path connecting $y_i$ to $x_i$ in $B(x_i, 2\ep_1),$ $i=1,2.$}
Let  ${\overline{X_{2\ep_1}}}=\{z\in \I^2: {\rm dist}(z, X)\le 2\ep_1\}.$ 
Choose $Y:=\mathtt{B}_{\overline{X}_{2\ep_1}}^k$ the $1/k$-brick cover 
of $\overline{X}_{2\ep_1}$ (see Definition \ref{Dbrickcover}).
Note that if $\overline{X}_{2\ep_1}$ is path connected then $Y$ is also connected.

Put $\ep_4=2\ep_1+2\sqrt{2}/k<2\ep_1+\ep_1/128.$  Then $X_{\ep_1}\subset Y\subset X_{\ep_4}.$

By Corollary \ref{C1C}, there is a unital injective \hm\, 
$\psi: C(Y)\to A/J$ such that
\beq\label{MT-pair-12}
\|\psi(e_j|_Y)-{\bar s}_j\|< \ep_4,\,\,\, 
j=1,2,\andeqn 
\psi_{*1}=\phi_{*1}\circ \pi^X_{*1},
\eneq
where $\pi^X: C(Y)\to C(X)$ is the quotient map by restriction. 
Put $s_j=\psi(e_j|_Y),$ $j=1,2,$ and 
$r=s_1+i s_2.$ Then ${\rm sp}(r)=Y.$ 

Denote by $\Omega_\infty$ the unbounded path connected component 
of  $\C\setminus X.$
Let $\Omega'\subset \C\setminus Y$ be a bounded path component. 
Then there are two cases:

(i) $\Omega'\subset \Omega_\infty;$

(ii)  if $z\in \Omega',$ then ${\rm dist}(z, X)\ge 2\ep_2.$

To see this, let us assume it is not the case (i).  Then $z$ is in a bounded path connected 
open component $\Omega\in \C\setminus X.$  But， since $X_{2\ep_2}\subset  X_{\ep_1}\subset Y,$ 
one must have  $z\not\in X_{2\ep_2}.$ 
Hence ${\rm dist}(z, X)\ge 2\ep_2.$ 

Next we will show that $\psi_{*1}=0.$
To see this, let $\lambda\in \Omega'\subset \C\setminus Y,$ where $\Omega'$
is a bounded path connected component of $\C\setminus Y.$
Put 
\beq
u=(\lambda-(e_1+i e_2)|_Y)([(\lambda-(e_1+i e_2)|_Y)^*(\lambda-(e_1+i e_2)|_Y)]^{-1/2})\in C(Y).
\eneq
If $\Omega'\subset \Omega_\infty$ (i.e., it is case (i)), then, for $\lambda\in \Omega',$
\beq
\pi^X_{*1}([u])=[u|_X]=0\,\,\,{\rm in}\,\,\, K_1(C(X)).
\eneq
Hence, 
by the second part of \eqref{MT-pair-12}, 
\beq
\psi_{*1}([u])=\phi_{*1}\circ \pi^X_{*1}([u])=0.
\eneq
If $\xi\in \Omega'$ and  ${\rm dist}(\lambda, X)\ge 2\ep_2$ (i.e., in the case (ii)),
by the second part of \eqref{MT-pair-12} and the claim above, one computes that
\beq
\psi_{*1}([u])=\phi_{*1}([u|_X])=\kappa_1(\lambda-({\bar s}_1+i {\bar s}_2))=0.
\eneq
 Hence
\beq
\psi_{*1}=0. 
\eneq
Since $Y\subset \C$ and  $K_0(C(Y))$ is finitely generated, $K_0(C(Y))$ is free.
Therefore $[\psi^w]=0$ in $KK(C(Y), A/J).$ 
Then, by Theorem \ref{T-1fs},  there is a normal element $r_0=t_1+i t_2\in  A/J$
with $t_1, t_2\in (A/J)_{s.a.}$
which has finitely many points in the spectrum contained in $Y$ such that
\beq\label{MT-pair-14}
\|r-r_0\|<\ep_2/16.
\eneq
Then, by  \eqref{MT-pair-7},  \eqref{MT-pair-12} and \eqref{MT-pair-14},
\beq
\hspace{-0.4in}\|\pi(T_1+iT_2)-r_0\| &\le &\|\pi(T_1+iT_2)-({\bar s_1}+i{\bar s_2})\|+\|({\bar s_1}+i{\bar s_2})-r\|+\|r-r_0\|\\
&<&\ep_2/2+\ep_4+\ep_2/16\le \ep_2/2+2\ep_1+\ep_1/128+\ep_2/16<3\ep_1.
\eneq
We may write $r_0=\sum_{l=1}^m \lambda_i p_i,$
where $\lambda_i\in Y$ and $p_1, p_2,...,p_m$ are mutually orthogonal non-zero projections in 
$A/J$ and $\sum_{k=1}^m p_k=1_{A/J}.$  
There are two cases: (1) $A$ has real rank zero, then, by Theorem 3.14 of \cite{BP},  every projection in $A/J$ lifts.
(2) If $Z$ is path connected, by 
the ``Moreover" part of Theorem \ref{T-1fs},
we may arrange such that $[p_k]=0$ or $[p_k]=[1_{A/J}]$ in $K_0(A/J).$
Since $A$ is unital, $[p_k]\in \pi_{*0}(K_0(A)).$ 
Thus, 
by  Lemma \ref{Lprojlift}, in both cases, 
there are mutually orthogonal projections 
$P_1, P_2,...,P_m\in A$ such that $\pi(P_j)=p_j,$ $j=1,2,...,m.$
There exists  $H\in {\cal K}$ such that
\beq\label{MT-pair-16}
\|(T_1+iT_2)+H-\sum_{j=1}^m \lambda_j P_j\|<\ep_1. 
\eneq
Then, there are projections $Q_j\in P_j{\cal K}P_j,$ $j=1,2,...,m,$ such that,
with $Q=\sum_{j=1}^m Q_j,$ 
\beq\label{MT-pair-17}
\|QH-HQ\|<\ep_1/16\andeqn \|(1-Q)H\|<\ep_1/16.
\eneq
Note that 
\beq\label{MT-pair-18}
Q(\sum_{j=1}^m \lambda_j P_j)=\sum_{j=1}^m \lambda_jQ_j=(\sum_{j=1}^m \lambda_j P_j)Q.
\eneq
Moreover, by  the second part of \eqref{MT-pair-17} and  \eqref{MT-pair-16}, and  then, by the first part 
of \eqref{MT-pair-17}, \eqref{MT-pair-16}, \eqref{MT-pair-18}, we obtain 
\beq\label{MT-p-40}
&&\hspace{-0.6in}\|(1-Q)(T_1+iT_2)(1-Q)-\sum_{j=1}^m \lambda_j (P_j-Q_j)\|<\ep_1/16+\ep_1=17\ep_1/16\\\label{MT-p-30}
&&\hspace{-0.4in}\andeqn  Q(T_1+iT_2)-(T_1+iT_2)Q\approx_{\ep_1/16} Q(T_1+iT_2+H)-(T_1+iT_2+H)Q\\\label{MT-p-31}
&&\hspace{1.7in}\approx_{2\ep_1} Q(\sum_{j=1}^m \lambda_jP_j)-(\sum_{j=1}^m \lambda_jP_j)Q=0.
\eneq
Hence
\beq\label{MT-p-41}
\|(T_1+iT_2)-((1-Q)(T_1+iT_2)(1-Q)+Q(T_1+iT_2)Q)\|<2(2+1/6)\ep_1
\eneq
It follows  from \eqref{MT-p-30} and \eqref{MT-p-31}, and by the assumption on $T_1$ and $T_2$ that
\beq
Q(T_1+iT_2)QQ(T_1-iT_2)Q\approx_{(2+1/6)\ep_1} Q(T_1+iT_2)(T_1-iT_2)Q\\
\approx_{2\dt} Q(T_1-iT_2)(T_1+iT_2)Q.
\eneq
Note $(2+1/16)\ep_1+2\dt< 3\ep_1+\ep_2<\dt_1.$
Note also that  $QJQ$ has property (IR).
By the choice of $\dt_1$ and  applying Theorem 4.4  of \cite{FR},
there is a normal element $ L\in QJQ$ such that
\beq
\|Q(T_1+iT_2)Q-L\|<\ep/16.
\eneq
Put $N=\sum_{j=1}^m \lambda_j(P_j-Q_j)+L.$ Then $N$ is normal
and, by  \eqref{MT-p-41} and  \eqref{MT-p-40}, 
\beq
\|(T_1+iT_2)-N\|<2(2+1/6)\ep_1+17\ep_1/16+\ep/16<\ep/8.
\eneq
Put  $S_1=(1/2)(N+N^*)$ and $S_2=(1/2i)(N-N^*).$ Then, since $N$ is normal,  $S_1S_2=S_2S_1.$ 
Moreover
\beq
\|S_j-T_j\|<\ep,\,\,\, j=1,2.
\eneq

\end{proof}

{\bf The Proof  of Theorem \ref{MT-pair}}.

\begin{proof}
For (i), choose $A=B(H)$ and $J={\cal K}.$ Note that ${\cal K}$ has real rank zero and stable rank one 
and has property (IR).  So Theorem \ref{MT-pair+} applies.

For (ii), let $A=L(H)$ and $J=K(H).$ 
If $K(H)$ is unital, then $L(H)=K(H).$
As in \ref{47}, this case follows from Theorem \ref{s2-T1}
(regardless of Fredholm index). 
Also in the proof of Theorem \ref{TTTmodule} (see \ref{47}), when $K(H)$ is not unital, 
it is a $\sigma$-unital purely infinite simple \CA\, which has real rank zero and has property (IR).
The fact that $A/J$ is purely infinite is also discussed in \ref{47} (the proof of Theorem \ref{TTTmodule}).
Thus Theorem \ref{MT-pair+} applies. 

For (iii), let $A=L(H)$ and $J=K(H).$ 
As in \ref{47} (the proof of Theorem \ref{TTTmodule}), by Kasparov's theorems,
$L(H)=M(K(H))$ and $K(H)$ is a $\sigma$-unital hereditary \SCA\, 
of a $\sigma$-unital simple \CA\, of real rank zero and stable rank one.
Hence $J=K(H)$ is also a simple \CA\, of real rank zero and stable rank one,
whence $J$ has property (IR).  If $K(H)$ is unital, then $J=K(H)=L(H).$
Therefore the conclusion follows from Theorem 4.4  of \cite{FR}. Otherwise, 
by the assumption, $M(J)/J$ is simple.  It follows that 
$J$ 
 has continuous scale (see of \cite{Linsc04}).
Hence, by Corollary 3.3 of \cite{Linsc04}, $A/J$ is a simple purely infinite simple \CA. 
Thus Theorem \ref{MT-pair+} applies.
\end{proof}



\begin{cor}\label{CCMT-p}
For  $\ep>0,$ there exists $\dt(\ep)>0$ satisfying the following:
Suppose that $H$ is an infinite dimensional separable Hilbert space and 
$T_1, T_2\in B(H)_{s.a.}$ 
with $\|T_i\|\le 1$ ($i=1,2$) such that
\beq\label{CMT-pair-00}
\|T_1T_2-T_2T_1\|<\dt\tand {\rm Ind}(\lambda-(T_1+iT_2))=0
\eneq
for all $\lambda\not\in \overline{({\rm sp}_{ess}(T_1+iT_2))_\dt}.$
Then there are $S_1, S_2\in B(H)_{s.a.}$
such that 
\beq
S_1S_2=S_2S_1\andeqn \|S_i-T_i\|<\ep,\,\,\, i=1,2.
\eneq
\end{cor}

 \begin{proof}
 Fix $\ep>0.$ Let $\dt_1$ be in Theorem \ref{MT-pair} for $\ep.$
 We may assume that $\dt_1<\ep.$
 Let $\dt:=\dt(\dt_1/4)$ be in Proposition \ref{Pesspsp} for $\dt_1/4$ (in place of $\eta$).
Then, if $T_1, T_2\in B(H)_{s.a.}$ satisfy the assumption \eqref{CMT-pair-00},
then, by Proposition \ref{Pesspsp} (see also \eqref{pSpdteta}). 
\beq
{\rm sp}_{ess}(T_1+iT_2)\subset s{\rm Sp}^{\dt_1/4}_{ess}((T_1, T_2))\subset 
(s{\rm Sp}^{\dt_1/4}_{ess}(T_1, T_2))_{\dt_1/4}\subset s{\rm Sp}^{\dt_1}_{ess}((T_1, T_2)).
\eneq
Then, for any $\lambda\not\in s{\rm Sp}^{\dt_1}_{ess}((T_1, T_2)),$
${\rm Ind}(\lambda-(T_1+iT_2))=0.$ So  
Theorem \ref{MT-pair} 
applies.
\end{proof}

The next proposition provides a  short way to see the second condition in \eqref{MT-pair-0} is necessary.
We also provide a direct proof.

\begin{prop}\label{Lnn-0}
Let  $1>\ep>0.$  There is $\dt>0$ satisfying the following:
Suppose that $T_1,T_2\in B(H)_{s.a.}$ and there are $S_1, S_2\in B(H)_{s.a.}$
such that
\beq
\|T_j-S_j\|<\dt,\,\,\,j=1,2, \tand S_1S_2=S_2S_1.
\eneq
Then, for any $\lambda\not \in sSp_{ess}^\ep(T_1+iT_2),$
\beq\label{Lnn-0-2}
{\rm Ind}(\lambda-(T_1+iT_2))=0.
\eneq
\end{prop}

Note  that \eqref{Lnn-0-2} also implies that $\lambda-(T_1+i T_2)$ is invertible.

\begin{proof}
Let $\dt_1:=\dt(2, \ep/4)$ be given by Proposition \ref{Pspuniq} for $\ep/4$ (in place of $\eta$).
Choose $\dt=\min\{\dt_1/4, \ep/9\}.$
Suppose that  $T_j$ and $S_j$ are in the statement of the theorem, $j=1,2.$
Put
$L=T_1+i T_2.$ 
and $N=S_1+iS_2.$ Then
\beq
\|L-N\|<2\dt.
\eneq
Since $N$ is normal, $X:={\rm sp}_{ess}(N)$ is an essential  $\dt_1$-near-spectrum of $L.$ 
By Proposition \ref{Pspuniq}, 
\beq
X\subset Z\subset X_{2\ep/4},
\eneq
where  $Z=sSp^{\ep/4}((\pi_c(T_1), \pi_c(T_2))).$ 
By \eqref{pSpdteta}, we may further have 
\beq\label{Lnn-0-8}
X\subset Z\subset \overline{Z_{\ep/4}}\subset  sSp^{\ep}_{ess}((T_1, T_2)).
\eneq
Since $N$ is normal,  
${\rm Ind}(\lambda-N)=0\rforal \lambda\not\in X.$
%
Note, for any $\lambda\in \C,$ 
\beq\label{Lnn-0-10}
\|(\lambda-\pi_c(L))-(\lambda-\pi_c(N))\|<2\dt<2\ep/9.
\eneq
Now if  $\lambda\not\in sSp^{\ep}_{ess}((T_1+iT_2)),$ then, by  \eqref{Lnn-0-8},
${\rm dist}(\lambda, \pi_C(N))\ge \ep/4.$
Since $N$ is normal,  
\beq\label{Lnn-0-11}
\|(\lambda-\pi_c(N))^{-1}\|={1\over{\rm dist}(\lambda, \pi_c(N))}.
\eneq
Hence, by \eqref{Lnn-0-10}, 
\beq\label{Lnn-0-5}
\|(\lambda-\pi_c(L))(\lambda-\pi_c(N))^{-1}-1\|<{2\dt\over{\ep/4}}<{8\over{9}}<1.
\eneq
This implies that $(\lambda-\pi_c(L))(\lambda-\pi_c(N))^{-1}$ is invertible and 
\beq\nonumber
(\lambda-\pi_c(L))(\lambda-\pi_c(N))^{-1}\in GL_0(B(H)/{\cal K}), \,{\rm or}\,\,\,
\kappa_1((\lambda-\pi_c(L))(\lambda-\pi_c(N))^{-1})=0.
\eneq
Since ${\rm Ind}(\lambda-N)=0,$ we conclude that
\beq\nonumber
{\rm Ind}((\lambda-\pi_c(L))=0.
\eneq
\end{proof}

\begin{rem}\label{RR}
Let us compare the results in this section and Theorem 1.1 of \cite{KS}:

While there is significant overlap between the results in this section and Theorem 1.1 of \cite{KS}, the proofs are quite different.  Theorem 1.1 of \cite{KS} provides a quantitative estimate for the distance from $T_1+iT_2$
​ to the set of normal elements with finite spectrum (though the explicit computation of 
$d_1(T_1+iT_2)$ may not be straightforward). In contrast, Theorem \ref{MT-pair+} here is a 
more conceptual result.
In the special case of Theorem  \ref{MT-pair} (i), the condition 
 $d_1(T_1+iT_2)=0$ (defined in \cite{KS}) is equivalent to
 ${\rm Ind}(\lambda-(T_1+iT_2))=0$ for all $\lambda\not\in {\rm sp}(T_1+iT_2).$ 
 
$\mathbf{\bullet}$ {\it Key Differences and Limitations:}

However, when  $K_1(A)\not=\{0\},$ or $K_1(J)\not=\{0\},$
 a normal element $x\in A$ (as in Theorem \ref{MT-pair+}) may not be approximable by normal elements with finite spectrum—even if $\kappa_1(\lambda-\pi(x))=0$ for all $\lambda\not\in {\rm sp}(\pi(x)).$
 In such cases, $d_1​(x)$ (from Theorem 1.1 of \cite{KS}) can be large. For instance, if $x$ is a unitary with $[x]\not=0$ in $K_1(A),$ 
  then $d_1(x)=1,$  rendering
  Theorem 1.1 of \cite{KS}  inapplicable (see Example \ref{Exm1} for details).

Moreover,  if a unital \CA\, $A$  has property (IR), then by Theorem 4.4 of \cite{FR},
 any pair of almost commuting self-adjoint elements $T_1$ and $T_2$ 
 is close to a commuting pair -- regardless of the size of $d_1(T_1+i T_2)$ 
 (see also Example \ref{Exm1} below).

$\mathbf{\bullet}${\it Computational Aspects of $\dt$-Synthetic Spectrum:}

From a computational standpoint, the $\dt$-synthetic spectrum $X$ is tractable: it consists of finitely many balls, and its complement has only finitely many bounded connected components. 
 Since the index remains constant in each such component, it suffices to test one point per bounded 
 components.
This avoids the complication of infinitely many ``holes" in the spectrum

$\mathbf{\bullet}$ {\it Further points:}

$\cdot$  The second condition in  \eqref{MT-pair-0} 
  does not imply  
 $d_1(T_1+iT_2)=0.$

$\cdot$   More importantly, it should be noted  that  Theorems \ref{TTTmul} and \ref{TTTmodule} here primarily address cases where $n>2,$  which lie outside the scope of \cite{KS}.

\end{rem}


\begin{exm}\label{Exm1}
Let $B_1$ be a unital separable purely infinite simple \CA\, with $K_0(B_1)=0$
and $K_1(B_1)\not=\{0\}$ and $J$ be a non-unital separable simple 
\CA\, with $K_1(J)\not=\{0\}.$ We assume that either $J$ is purely infinite simple, or
$J$ is a separable simple \CA\, of real rank zero and stable rank one.
Consider any unital \CA\, $A$ which is given by  the following essential extension:
\beq\nonumber
0\to J\to A\to B_1\to 0.
\eneq
Let $\pi: A\to B_1$ be the quotient map.
Note that $\pi(A)\cong A/J\cong B_1$ is purely infinite, $J$ has property (IR).
Since $K_0(B_1)=0,$ every projection in $B_1=A/J$ lifts to a projection in $A$ (see Lemma \ref{Lprojlift}). It follows from Corollary  3.16 of  \cite{BP} 
that $A$ has real rank zero. 
Hence Theorem \ref{MT-pair+} applies in this case. 
However, from the six-term exact sequence in 
$K$-theory, one computes that $K_1(A)\not=\{0\}.$
Denote  by $N(A)$ the set of normal elements of $A$  and 
$N_f(A)$ the set of those elements in $N(A)$ which has finite spectrum, respectively.
Then the closure $\overline{N_f(A)}\not=N(A).$

 Let $p\in J$ be a non-zero  projection.
Since $pJp$ is either purely infinite or stable rank one, 
$K_1(pAp)=K_1(pJp)\cong K_1(J)\not=\{0\}.$ Thus there is a unitary $u_0\in pJp$
with $[u_0]\not=\{0\}.$ Let $x=(1/4)(1-p)+u_0.$ Then $x$ is a normal element  with
$\pi(x)=(1/4)1_{A/J}.$ Hence $\kappa_1(\lambda-\pi(x))=0$  for any 
$\lambda\not\in {\rm sp}(\pi(x))=\{1/4\}$ (so  Theorem \ref{MT-pair} (iii) applies). 

Then ${\rm sp}(x)=\T\cup \{1/4\}.$
Choose $d=1/32.$ Put $X=\{z\in \C: {\rm dist}(z, {\rm sp}(x))\le d\}.$
Suppose that $y\in A$  with $\|x-y\|<d/2.$ 
Then, for any $\lambda\not\in X,$
\beq
\|1-(\lambda-x)^{-1} (\lambda-y)\|<{d/2\over{d}}=1/2.
\eneq
It follows that $\lambda\not\in {\rm sp}(y).$  In other words, 
${\rm sp}(y)\subset X.$ 

Let $D$ be the  disk in $\C$ with center at $0$ and radius $3/2.$  
Then $X\subset  D.$ 
Choose   $g\in C([0, 3/2])_+$ such that 
$g(r)=1,$ if $r\in [1-d, 3/2],$ $0\le g(r)\le 1,$ if $1/4+d< r\le 1,$ $g(r)=0$
if $r\in [0, 1/4+d].$  Define 
$f\in C(D)$ such that $f(re^{2\pi i \theta})= g(r)e^{2\pi i \theta}$ ($\theta\in [0,1)$).
Note that $f(X)\subset \T\cup \{0\}$ and 
$f(0)=0.$ Then, there exists $0<\dt_1<1/2$ 
such that, for any pair of normal elements $y_1, y_2$ with $\|y_i\|\le 3/2$
and with $\|y_1-y_2\|<\dt_1,$ then 
\beq
\|f(y_1)-f(y_2)\|<1/16.
\eneq
Put $\dt=\min\{d, \dt_1\}>0.$ 

Claim: 
${\rm dist}(x, N_f(A))\ge \dt$  but ${\rm dist}(x, N(A))=0.$

In fact, suppose otherwise there is  $y\in N_f$ such that
\beq
\|x-y\|<\dt.
\eneq
Then, since $\dt\le d,$ ${\rm sp}(y)\subset X,$ and 
\beq
\|f(x)-f(y)\|<1/16.
\eneq
Hence $f(x)=u_0$ and ${\rm sp}(f(y))\subset \T\cup \{0\}.$ 
Therefore $f(y)$ is a normal partial isometry. 
Moreover,
\beq
\|p-f(y)f(y)^*\|\le \|u_0u_0^*-f(y)u_0^*\|+\|f(y)u_0^*-f(y)f(y)^*\|<1/8.
\eneq
There exists a unitray $w\in A$ such that
\beq
\|w-1\|<1/4\andeqn w^*f(y)f(y)^*w=p.
\eneq
Also
\beq\nonumber
\|u_0-w^*f(y)w\| &\le& \|u_0-f(y)\|+\|f(y)-w^*f(y)\|+\|w^*f(y)-w^*f(y)w\|\\\nonumber
&<&1/16+1/4+1/4<1.
\eneq
Note that $w^*f(y)w$ is a unitary in $pJp$ with finite spectrum, hence 
$w^*f(y)w\in U_0(pJp).$ This contradict the fact that $[u_0]\not=0$ in $pJp,$
which proves the claim. 

  In other words, 
Theorem 1.1 in \cite{KS} could not be applied to determine 
${\rm dist}(s_1+is_2, N(A)),$ when $s_1$ and $s_2$ are almost commuting
 (see also  Theorem in  \cite{LinWFN}).
\end{exm}

\section{Epilogue}

In this final section, we examine the connection between the vanishing Fredholm index condition and the closeness of the (synthetic) spectrum to the essential (synthetic) spectrum (Corollaries \ref{Chind=sp} and \ref{Cind=sp-3}).
 We also present an example (Proposition \ref{PLast}) where a higher-dimensional analogue of Theorem \ref{MT-pair}  fails-- highlighting the  obstructions and subtleties behind this failure. Finally, we conclude with a remark on Mumford’s original question regarding quantum theory and measurements, relating it to the second condition in Theorem  \ref{TTTmul}.

Let us first state the following corollary of Theorem \ref{TTmul}.

\begin{cor}\label{CCC}
Let $n\in \N$ and $\ep>0.$ 
There exists $\dt(n, \ep)>0$ satisfying the following:
Suppose that $T_1, T_2,...,T_n\in B(H)_{s.a.}$ with $\|T_i\|\le 1$ ($1\le i\le 1$) such that
\beq
\|T_iT_j-T_jT_i\|<\dt,\,\,\, 1\le i,j\le n
\eneq
and $Y=s{\rm Sp}^{\ep/4}_{ess}((T_1, T_2,...,T_n))\supset \I^n.$
Then there are $S_1, S_2,...,S_n\in B(H)_{s.a.}$
such that
\beq
S_iS_j=S_jS_i,\,\,\, 1\le i,j\le n\andeqn \|T_j-S_j\|<\ep,\,\,\, 1\le j\le n.
\eneq

\end{cor}

\begin{proof}
This follows from Theorem \ref{TTmul} and the fact that 
\beq\nonumber
s{\rm Sp}_{ess}^{\ep/4}((T_1, T_2, ...,T_n))\subset s{\rm Sp}^{\ep/4}((a_1,a_2,...,a_n))\subset \I^n.
\eneq
\end{proof}

The next corollary shows that, when $n=2,$  the second condition in \eqref{TTmul-0}  is stronger than 
that of vanishing index  in Theorem \ref{MT-pair}. 

\begin{cor}\label{Chind=sp}
For any $\eta>0,$ there exists $\dt>0$ 
 satisfying the following:
  Let $A$ and $H$ be in all cases (i), (ii) and (iii) of Theorem \ref{MT-pair}.
Suppose that $T_1, T_2\in L(H)_{s.a.},$ 
\beq\label{Chind-sp-0}
\|T_1T_2-T_2T_1\|<\dt\andeqn
d_H(s{\rm Sp}^{\eta}((T_1, T_2)), s{\rm Sp}^\eta((\pi(T_1), \pi(T_2)))<\eta.
\eneq
Then, for any $\lambda\not\in  s{\rm Sp}^{4\eta}((\pi(T_1), \pi(T_2))),$
\beq
{\rm \kappa_1}(\lambda-(\pi(T_1+iT_2)))=0.
\eneq
\end{cor}

\begin{proof}
Let $\dt_1:=\dt(\eta/4)>0$ be given by  Proposition \ref{Pesspsp} for $\eta/4.$

We choose $\dt_2>0$ such that the conclusion of Proposition \ref{Pspuniq} holds for $\min\{\dt_1/2, \eta/4\}$ (in place of $\eta$). 
Then we choose $\dt_3>0$ such that the conclusion of Proposition \ref{Pappsp} holds for $\dt_2/2.$
Choose $\dt=\dt_3.$ Suppose that \eqref{Chind-sp-0} holds.
Put $L=T_1+iT_2.$
Thus  (applying Proposition \ref{Pappsp}) $Z:=s{\rm Sp}^{\dt_1}((T_1, T_2)), Z_e:=s{\rm Sp}^{\dt_1}(\pi(T_1), \pi(T_2)),$ 
 $X:=nSp^{\dt_2}((T_1, T_2)),$ and $X_e:=nSp^{\dt_2}(
\pi(T_1), \pi(T_2))$
 are all non-empty sets.
By Proposition \ref{Pspuniq}, 
\beq
X\subset Z\subset X_{\dt_1}\andeqn X_e\subset Z_e\subset (X_e)_{\dt_1}.
\eneq
By applying Proposition \ref{Pesspsp}, we  have 
\beq
{\rm sp}(T_1+iT_2)\subset s{\rm Sp}^{\eta/4}((T_1, T_2)) \andeqn {\rm sp}(\pi(T_1+iT_2))\subset s{\rm Sp}^{\eta/4}((\pi(T_1),\pi(T_2))).
\eneq
Applying \eqref{pSpdteta}, we obtain 
\beq
{\rm sp}(T_1+iT_2)\subset (s{\rm Sp}^{\eta/4}((T_1, T_2))_{\eta/4}\subset s{\rm Sp}^{\eta}((T_1, T_2)).
\eneq
Since 
\beq
d_H(s{\rm Sp}^\eta((T_1,T_2)), s{\rm Sp}^{\eta}(\pi(T_1), \pi(T_2)))<\eta,
\eneq
one has  (see \eqref{pSpdteta})
\beq
{\rm sp}(T_1+iT_2)\subset \overline{s{\rm Sp}^{\eta}(\pi(T_1), \pi(T_2)))_{\eta}}\subset s{\rm Sp}^{4\eta}((\pi(T_1), \pi(T_2))).
\eneq
So, if $\lambda\not\in s{\rm Sp}^{4\eta}((\pi(T_1), \pi(T_2))),$ then
$\lambda\not\in {\rm sp}(T_1+iT_2).$ 
By \cite{CH}, in case (i) and (ii), since $A$ is stable, $K_1(L(H))=K_1(M(A))=\{0\}.$
For case (iii), $K_1(L(H))=K_1(M(A))=0$ follows from Lemma 3.3 of \cite{Linscand95} (note that, 
proved in \cite{LinBP}, if $C$ has real rank zero, then ${\rm cer}(C)\le 1+\ep\le 2$).
Therefore
\beq
\kappa_1(\lambda+(T_1+iT_2))=0.
\eneq
Hence 
\beq
\kappa_1(\lambda-\pi(T_1+iT_2))=0\,\,\, {\rm in}\,\,\, K_1(L(H)/K(H)).
\eneq
\end{proof}


In general, we have the following statement, which follows immediately from Theorem \ref{TTmul}.  Following the same spirit as the proof of Corollary \ref{Chind=sp},  one could alternatively give a direct proof without invoking Theorem \ref{TTmul}.  However, since this approach is already demonstrated in Corollary \ref{Chind=sp},  we avoid introducing additional $KK$-theoretic technicalities -- particularly to spare readers less familiar with $KK$-theory.

\begin{cor}\label{Cind=sp-3}
Let $n\in \N.$ For any $\ep>0,$ there exists $\dt(n, \ep)>0$ satisfying the following:
Suppose that $T_1,T_2,...,T_n\in B(H)_{s.a.}$ with $\|T_i\|\le 1$ ($i=1,2$) 
such that
\beq
\|T_iT_j-T_jT_i\|<\dt,\,\,\, 1\le i, j \le n, \tand d_H(X, Y)<\ep/8,
\eneq
where $X=s{\rm Sp}^{\ep/8}((T_1, T_2,...,T_n))$ and 
$Y=s{\rm Sp}_{ess}^{\ep/8}((T_1,T_2,...,T_n)).$ 
Then there exists a  $1/k$-brick combination $\Omega\subset \I^n$ and a unital monomorphism 
$\phi: C(\Omega)\to B(H)/{\cal K}$ 
such that 
\beq
&&\|\phi(e_j|_\Omega)-\pi_c(T_j)\|<\ep,\,\,\, 1\le j\le n, \andeqn\\
&& [\phi^\omega]=0\,\,\,{\rm in}\,\,\, KK(C(\Omega), B(H)/{\cal K}).
\eneq
\end{cor}

\begin{proof}
Applying Theorem \ref{TTmul}, one obtains commuting 
$n$-tuple of self-adjoint operators $S_1,S_2,...,S_n\in B(H).$
Let $C=C(\Omega)$ be the \SCA\, of $B(H)$ generated by
$1, S_1, S_2, ...,S_n.$ Define $\psi: C(\Omega)\to B(H)$ to be the embedding.
Since $K_i(B(H)\otimes C(Y))=0$ for any compact metric space $Y,$ $i=0,1,$
we conclude that $KK([\psi])=0.$ Consequently $[\phi^\omega]=0,$
where $\phi=\pi_c\circ \psi.$
\end{proof}

We now present a counterexample demonstrating  that the $n$-dimensional analogue  of   
Theorem \ref{MT-pair}  (i)  is false when $n\ge 3.$  Note that, in the next proposition, 
the triple $(T_1, T_2,T_3)$  has only a  single point in the essential $\eta$-spectrum. 
Hence the $KK$-theory obstruction vanishes in $B(H)/{\cal K}.$  This example 
also underscores the significance  of Theorems \ref{TTTmul} and \ref{TTTmodule}.

\begin{prop}\label{PLast}
There exists $0<\ep_0<1$ and a sequence of 
triples of self-adjoint operators $T_{n,1}, T_{n,2},T_{n,3}\in B(H)$
such that $\|T_{n,j}\|\le 1$ (for all $n$),
\beq\nonumber
\lim_{n\to\infty}\|T_{n,i}T_{n,j}-T_{n,j}T_{n,i}\|=0,\,\,\, 
\pi_c(T_j)=1_{B(H)/{\cal K}},\,\,\,\, 1\le i,j\le 3, \tand\\\nonumber
\inf\{\max_{1\le j\le 3}\{\|T_{n,j}-S_{n,j}\|: S_{n,j}\in B(H)_{s.a.} \andeqn S_{n,j}S_{n,i}=S_{n,i}S_{n,j},\,\,1\le i,j\le 3\}\}\ge 
\ep_0,
\eneq
where infimum is taken from all possible commuting triples $\{S_{n,j}: 1\le j\le 3\}$ of selfadjoint operators and all $n\in \N.$
\end{prop}

\begin{proof}
Let $\D^3$ be the three dimensional unit ball in $\I^3.$ 
By Theorem 4.1 of \cite{GL2}, there exists a sequence of contractive linear maps 
$\Lambda_n: C(\D^3)\to M_{n^3}$ and    $a>0$ which satisfy the following:
\beq\label{Plast-4}
&&\lim_{n\to \infty}\|\Lambda_n(fg)-\Lambda_n(f)\Lambda_n(g)\|=0\rforal f, g\in C(\D^3)\andeqn\\
\label{PLast-5}
&&\inf \{\sup\{\|\Lambda_n(e_j|_{\D^3})-\phi_n(e_j|_{\D^3})\|: 0\le j\le 3\}\}\ge a, 
\eneq
where the infimum is taken among all \hm  s $\phi_n: C(\D^3)\to M_{n^3}.$

Let  $B_1=\prod_nM_{n^3}$ and $\Pi_1: B_1\to B_1/\bigoplus_n M_{n^3}$ be  the quotient map. 
Then $\psi=\Pi_1\circ \{\Lambda_n\}: C(\D^3)\to B_1/B_1/\bigoplus_nM_{n^3}$ 
is a \hm.   Theorem 4.1 of \cite{GL2} allows each $\Lambda_n$ to be 
$\sigma_n$-injective (see Definition 0.3 in \cite{GL2}) with $\sigma_n\to 0.$ This feature implies 
that $\psi$ is injective. Let $q=\psi(1).$ Then 
there exists a sequence of projections 
$p_n\in M_{n^3}$ such that $\Pi_1(\{p_n\})=q.$ 
Put $D_n=p_nM_{n^3}p_n,$ $n\in \N.$
Note that
\beq
\lim_{n\to \infty}\|\Lambda_n(1)-p_n\|=0.
\eneq
Hence 
\beq
\lim_{n\to\infty}\|p_n\Lambda_n(1)p_n-p_n\|=0.
\eneq
There are $d_n\in D_n$ such that, for all sufficiently large $n,$ 
\beq
d_np_n\Lambda_n(1)p_nd_n=p_n \andeqn \lim_{n\to\infty}\|d_n-p_n\|=0.
\eneq
Therefore, by replacing $\Lambda_n$ with $d_np_n\Lambda_np_nd_n,$
we may assume that $\Lambda_n$ maps  $C(\D^3)$ into  $D_n$ and 
each $\Lambda_n$ is unital.

Choose $\lambda_0=(1,1,1)\in \I^3.$ 
Identify  $M_{n^3}$ with a \SCA\, of ${\cal K}\subset B(H)$ such that
$1_{M_{n^3}}$ becomes a projection of rank $n^3.$ 
In particular, $p_n\in B(H)$ as  a finite rank ($\le n^3$) projection.
Put $X=\D^3\sqcup \{(1,1,1)\}.$ 
Define $\Phi_n:C(X)\to B(H)$ by
\beq
\Phi_n(f)=f((1,1,1))(1-p_n)+\Lambda_n(f|_{\D^3})\rforal f\in C(\I^3).
\eneq
Choose  (with $e_0=1$)
\beq
T_{n,j}=\Phi_n(e_j)=(1-p_n)+\Lambda_n(e_j|_{\D^3}),\,\,\,j=0,1,2, 3.
\eneq
So $T_{n,j}\in B(H)_{s.a.}$ and $\|T_{n,j}\|\le 1,$ $0\le j\le 3.$
Moreover 
\beq
\lim_{n\to\infty}\|T_{n,i}T_{n,j}-T_{n,j}T_{n,i}\|=0,\,\,\, 1\le i,j\le 3.
\eneq
Furthermore, 
\beq\label{Plast-10}
\pi_c(T_{n,j})=\pi_c(e_j((1,1,1))(1-p_n))=e_j((1,1,1))1_{B(H)/{\cal K}}=1_{B(H)/{\cal K}},\,\, \, 1\le j\le 3,
\eneq
where $\pi_c: B(H)\to B(H)/{\cal K}$ is the quotient map.

Let $B=\prod(\{B(H)\})$ and $\Pi: B\to B/\bigoplus(\{B(H)\})$ be the quotient map.
Define $\td \phi: C(X)\to B/\bigoplus(\{B(H)\})$ by
$\td \phi(f)=f((1,1,1))(1-q)+\phi(f|_{\D^3})$ for $f\in C(X).$

If there is a sequence of  commuting triples of self-adjoint operators $S_{n,j}\in B(H)$
such that 
\beq
\lim_{n\to\infty}\|S_{n,j}-T_{n,j}\|=0,\,\,\, 1\le j\le 3, 
\eneq
then one obtains a sequence of  unital \hm s $\psi_n: C(\I^3)\to B(H)$ 
such that
\beq
\lim_{n\to\infty}\|\psi_n(e_j)-T_{n,j}\|=0,\,\,\, 1\le j\le 3.
\eneq
Hence 
\beq\label{Plast-15}
\lim_{n\to\infty}\|\psi_n(f)-\Phi_n(f|_X)\|=0\rforal f\in C(\I^3).
\eneq
Then $\Pi\circ\{ \psi_n\}=\td \phi\circ r,$
where $r: C(\I^3)\to C(X)$ is defined by $r(f)=f|_X$ for $f\in C(\I^3).$

Choose $f_0\in C(\I^3)$ such that 
$0\le f_0\le 1,$ $f_0(\xi)=0,$ if $\dist(\xi, \D^3)\ge 1/32$ and $f_n(\xi)=1,$
if $\dist(\xi, \D^3)<1/64.$ 
Note that $f_0((1,1,1))=0$ and 
\beq 
\Phi_n(f_0|_X)=p_n \andeqn \Phi_n(f_0e_j|_X)=\Lambda_n(e_j|_X),\,\, 1\le j\le 3,\,\,\,n\in \N. 
\eneq
Put  $c_{n,j}=\psi_n(f_0e_j),$ $1\le j\le 3$  and 
$c_{n,0}=p_n,$ $n\in \N.$
Then 
\beq
\lim_{n\to\infty} \|c_{n,j}-\Lambda_n(e_j|_{\D^3})\|=0\tforal 1\le j\le 3. 
\eneq
Recall that 
\beq
\sqrt{\|e_1^2|_{\D^3}+e_2^2|_{\D^3}+e_3^2|_{\D^3}\|}\le 1.
\eneq
Put ${\bar c}_{n,j}={c_{n,j}\over{\|\sum_{j=1}^3 c_{n,j}^2\|^{1/2}}},$ $1\le j\le 3,$
${\bar c}_{n,0}=p_n,$ and $n\in \N.$
It follows that
\beq
\lim_{n\to\infty}\|{\bar c}_{n,j}-\Lambda_n(e_j|_{\D^3})\|=0.
\eneq
Let $C_n$ be the (commutative) \SCA\, of $B(H)$ generated by 
$\{{\bar c}_{n,j},$ $0\le j\le 3\}.$ Then there is an isomorphism 
$h_n: C_n\cong C(Y_n)$ for some compact subset $Y_n\subset \D^3$
such that  $h_n^{-1}(e_j|_{Y_n})={\bar c}_{n,j},$ $0\le j\le 3,$ and $n\in \N.$ 
Define $H_n:C(\D^3)\to C\subset B(H)$ by
$H_n(f)=h_n^{-1}(f|_{Y_n})$ for $f\in C(\D^3).$ 
Since $\{e_j|_{\D^3}:0\le j\le 3\}$ generates $C(\D^3),$
we have 
\beq
\lim_{n\to\infty}\|H_n(f)-\Lambda_n(f)\|=0\rforal f\in C(\D^3).
\eneq
Let $E_n=H_n(1_{C(\D^3)}),$ $n\in\N.$ Then,
\beq
\lim_{n\to\infty}\|E_n-p_n\|=0.
\eneq
 Thus, for all large $n,$ 
there exist unitaries $w_n\in B(H)$ such that
\beq
w_n^*p_nw_n=E_n,  w_nE_nw_n^*=p_n\andeqn \lim_{n\to\infty}\|w_n-1\|=0.
\eneq
Define $\Psi_n: C(\D^3)\to p_nB(H)p_n$ 
by $\Psi_n(f)=w_nH_n(f)w_n^*$ for $f\in C(\D^3).$ 
Then
\beq
\lim_{n\to\infty}\|\Psi_n(f)-\Lambda_n(f)\|=0\rforal f\in C(\D^3).
\eneq
This contradicts  \eqref{PLast-5}. Therefore  there is  no
\hm s $\psi_n$ such that
\eqref{Plast-15} holds.  In other words, there is $\ep_0>0$ such that 
\beq
\inf \{\sup\{\|\psi_n(e_j)-T_{n,j}\|: 1\le j\le 3\}\}\ge \ep_0, 
\eneq
where the infimum is taken among all \hm  s $\psi_n: C(\I^3)\to B(H).$
Combining  this with \eqref{Plast-10}, the proposition follows.
\end{proof}

\begin{rem}\label{Rcomp}
By 
Lin's theorem 
for almost commuting self-adjoint matrices
(\cite{Linmatrices}, see also \cite{FR}), when $n=2,$ 
the answer to the version of the Mumford problem (\eqref{mfp-1} and \eqref{mfp-2}) 
is affirmative if each $T_j$ is compact. However, if $n>2,$ the answer is negative 
even each $T_j$ is compact in general, by Theorem 4.1 of \cite{GL2}.
\end{rem}

\begin{rem}\label{Finalrem}
This remark was added to the original version of the paper after we 
had the opportunity to study Chapter 14 of David Mumford's insightful new book \cite{Mf2},
kindly provided by the author. 
%
Some terminologies and motivation are drawn from this work. 

In quantum theory, macroscopic observables may be modeled as 
self-adjoint operators  $T_1, T_2,...,T_n$ on an infinite dimensional  separable Hilbert space $H.$
Let $A$ be the  unital \SCA\, generated by $\{T_j: 1\le j\le n\}$ and $1$
(assuming $\|T_j\|\le 1$ for convenience).
Suppose that $\{e_m\}$ is an orthonormal basis of $H,$  and 
suppose that, for any $\ep>0,$ there is $N\in \N$ such that
for all $m\ge N,$ 
$
\|T_je_m\|<\ep.
$
In this case these observables collapse uniformly and  any measurements 
outside (orthogonal to) a finite dimensional subspace vanish, meaning an external observer would detect nothing. We exculde this case by assuming 
none of $T_j$'s are compact (see also Remark \ref{Rcomp}).

Given a state defined by a unit vector $x\in H,$ the expected value of observable $T$ 
 is $\exp_T(x)=\la Tx,x\ra$ (see p.178 of \cite{Mf2}).
 Mumford introduces {\it Approximately Macroscopically Unique (AMU) states} as ``near eigenvectors":
for a small $\sigma>0,$ 
$$
{\rm AMU}(\{T_j: 1\le j\le n, \sigma\})=\{x\in H: \|x\|=1, {\rm sd}_{T_j}(x)<\sigma,\,\,1\le j\le n\},
$$
where ${\rm sd}_{T_j}(x)=\|(T_j-\exp_{T_j}(x) I)x\|$ (see Chapter 14 of \cite{Mf2}). 
Mumford asked whether AMU exist. 
As noted   in \cite{Mf2}, for AMU to be non-empty, it is natural to require 
the commutators $\|[T_i,\, T_j]\|$ to be small.  
If the question posed in the introduction (see \eqref{mfp-1} and \eqref{mfp-2}) had an affirmative
answer,  then the  quantum system could be approximated by a classical one, ensuring the existence 
of AMU states. 
However, as  
observed by Mumford,  small commutators alone are insufficient in general.
(Proposition \ref{PLast} further shows the complexity of the problem).

In an infinite dimensional separable Hilbert space $H,$ a self-adjoint operator  such as $T_j$ may
lack 
eigenvalues but 
admit approximate eigenvalues.
  However, if one allows non-normal states of $B(H),$  or that of $A+{\cal K}$
  (where ${\cal K}$ is the ideal of compact operators),  then 
$T_j$ does have eigenstates and eigenvalues in the atomic representation of $B(H),$ or $A+{\cal K}$ 
(see, for example, Proposition 4.3.10 of \cite{Pedbook}). In particular,  one should consider those  states of $B(H)$ 
which vanish on ${\cal K}.$ 

For a  point $\xi=(t_1, t_2,...,t_n)$ in the unit cube of $\R^n$ to be a joint approximate eigenvalue, 
we may test it using the function $\Theta_{\xi,\eta}$ defined in \ref{Dpseudosp}. 
%
In order to have $\|\Theta_{\xi, \eta}(T_1, T_2,...,T_n)\|\ge 1-\eta$ (for small $\eta$),
one only needs one vector state $\phi(\cdot)=\la \,\cdot\,\, x_1, x_1\ra$ (with a unit vector $x_1\in H$) such that 
\beq\label{int-1}
\la \Theta_{\xi,\eta}(T_1, T_2,...,T_n)x_1, x_1\ra\ge 1-\eta.
\eneq
One may expect this can be repeatedly measured.
Suppose that $y_1,...,y_{l_1}$ are unit vectors and $H_1$ is a finite dimensional subspace spanned by
$x_1, y_1,...,y_{l_1}.$ Suppose that there is another unit vector $x_2\in H_1^\perp$
such that \eqref{int-1} also holds for $x_2$ (in place of $x_1$), and this persists. 
Then we obtain a sequence of $\{x_m\}$ ($x_{m+1}\perp H_m$) such that 
\eqref{int-1} holds for each $x_m$ (in place of $x_1$).
Note that unit vector states given by
$\{x_m\}$ has limit points which give   states $\phi$ of $A+{\cal K}$ vanishing
on ${\cal K}$ and  $\phi(\Theta_{\xi,\eta}(T_1, T_2,...,T_n))\ge 1-\eta.$

Let $\pi_c: B(H)\to B(H)/{\cal K}$ be the quotient map and $\psi$ be a state of 
the corona algebra $B(H)/{\cal K}$ (or  that of $(A+{\cal K})/{\cal K}$).  Then
the composition $\psi\circ \pi_c$ provides a 
 quantum state which may be better suitable to this setting.
Indeed, one may expect
at least for one such quantum state $\psi\circ \pi_c$ satisfying
\beq\label{int-2}
\psi\circ \pi_c (\Theta_{\xi,\eta}(T_1, T_2,...,T_n))\ge 1-\eta.
\eneq
Such measurements, 
given by states of the form $\psi\circ \pi_c,$ may be called  {\it outside measurements}.
Via  the Gelfand-Naimark construction, this state $\psi\circ \pi_c$  corresponds to a vector state $y$ in the atomic representation space $H_a,$ yielding
 \beq\label{int-3}
\la \Theta_{\xi,\eta}(T_1, T_2,...,T_n)y,y\ra\ge 1-\eta. 
 \eneq
 If inequality \eqref{int-2}  holds
  whenever \eqref{int-1} does, then 
the second condition in \eqref{TTmul00-1} always satisfies (in fact, in this case $X=Y$). 
 
 Let us temporerly propose the following terminology:
 
 $\cdot$ A unit vector $x\in H$ satisfying \eqref{int-1} is a 
{\it synthetic-AMU state}, with $\xi$ a {\it synthetic-joint-expected value.}

$\cdot$  A state of the form $\psi\circ\pi_c$ satisfying  \eqref{int-2} is an 
{\it essential synthetic-AMU state}, with $\xi$ {\it an essential synthetic-joint-expected value}.

Note that 
any  AMU state (with sufficiently small $\sigma$) is an ($\eta$)-synthetic-AMU state. 
Moreover, by Proposition \ref{Pappsp}, the set of synthetic-AMU states is non-empty,
when $\|[T_i, T_j]\|$ is small (independent of $T_j$).
The second condition in Theorem \ref{TTTmul} then may be reformulated as:

(C): {\it all synthetic-joint-expected values lie  in a sufficiently small neighborhood of the set 
of essential synthetic-joint-expected values.}

This seems to suggest  a physical constraint on macroscopic systems: 
{\it local measurement should not deviate significantly from outside measurements} 
(or a long-term measurements along an orthonormal basis of the Hilbert space). 
Under this constraint, when commutators tend to zero, by Theorem \ref{TTTmul}, the quantum system 
becomes approximable by a 
classical commutative one.

\end{rem}

\vspace{0.4in}



\noindent email: hlin@uoregon.edu
\end{document}